\documentclass[11pt,reqno,twoside]{amsart}
\usepackage{tikz}
\usepackage{amssymb}
\usepackage{amsmath}
\usepackage{amsthm}
\usepackage{amsfonts}
\usepackage{amsthm,amscd}
\usepackage{amsbsy}
\usepackage{bm}
\usepackage{url} 
\usepackage{setspace}
\usepackage{float}
\usepackage{psfrag}
\usepackage{tikz}
\usetikzlibrary{calc,intersections}
\usepackage{float}
\usepackage{epstopdf}
\usepackage{caption}

\usetikzlibrary{positioning}

\usepackage[colorlinks, linkcolor=blue]{hyperref}

\restylefloat{figure}

\newtheorem{Theorem}{Theorem}[section]
\newtheorem{Lemma}[Theorem]{Lemma}

\newtheorem{Proposition}[Theorem]{Proposition}
\newtheorem{Remark}[Theorem]{Remark}

\numberwithin{equation}{section}
\usepackage{setspace}
\usepackage{anysize}

\usepackage[margin=1in]{geometry}
\usepackage[noadjust]{cite}

\def\be{\begin{equation}}
	\def\ee{\end{equation}}
\def\ben{\begin{eqnarray}}
	\def\een{\end{eqnarray}}

\hypersetup{colorlinks=true,%
	citecolor=red,%
	filecolor=blue,%
	linkcolor=blue,%
}

\newcommand{\ncom}{\newcommand}
\usepackage{mathtools}
\mathtoolsset{showonlyrefs=true}

\ncom{\n}{\normalfont}

\ncom{\N}{\mathbb{N}}
\ncom{\Lc}{\mathcal}
\ncom{\wt}{\widetilde}
\ncom{\Af}{\boldsymbol{A}}
\ncom{\Bf}{\boldsymbol{B}}
\ncom{\Hf}{\mathbf{H}}
\ncom{\Pf}{\boldsymbol{P}}

\ncom{\ui}{\boldsymbol{u}_\infty}
\ncom{\vi}{\boldsymbol{v}_\infty}
\ncom{\wi}{\boldsymbol{w}_\infty}
\ncom{\yi}{\boldsymbol{y}_\infty}
\ncom{\zi}{\boldsymbol{z}_\infty}
\ncom{\fbi}{\boldsymbol{f}_\infty}
\ncom{\pbi}{\boldsymbol{p}_\infty}

\ncom{\Lb}{\mathbb{L}^2(\Omega)}
\ncom{\Hoz}{\mathbb{H}^1_0(\Omega)}
\ncom{\Hzz}{\mathbb{H}^2(\Omega)}
\ncom{\Vf}{\mathbb{V}}
\ncom{\no}{\nonumber}
\ncom{\ub}{\boldsymbol{u}}
\ncom{\yb}{\boldsymbol{y}}
\ncom{\zb}{\boldsymbol{z}}
\ncom{\vb}{\boldsymbol{v}}
\ncom{\fb}{\boldsymbol{f}}
\ncom{\wb}{\boldsymbol{w}}

\ncom{\T}{\mathbb{T}}
\ncom{\C}{\mathbb{C}} 
\ncom{\Hb}{\mathbb{H}}
\ncom{\Pb}{\mathbb{P}}
\ncom{\V}{\mathbb{V}}
\ncom{\U}{\mathbb{U}}
\ncom{\Ac}{\mathcal{A}}
\ncom{\Bc}{\mathcal{B}}
\ncom{\Pc}{\mathcal{P}}
\ncom{\af}{\boldsymbol{a}}
\ncom{\pf}{\boldsymbol{p}}

\newcommand{\vertiii}[1]{{\left\vert\kern-0.25ex\left\vert\kern-0.25ex\left\vert #1 
		\right\vert\kern-0.25ex\right\vert\kern-0.25ex\right\vert}}

\long\def\/*#1*/{}
\title[Stabilizability of 2D and 3D Navier-Stokes equations]{Stabilizability of 2D and 3D Navier-Stokes equations with memory around a non-constant steady state}
\date{\today}

\author{{WASIM AKRAM$^\dag$, Manika Bag$^\ddag$,}
	\and{Manil T. Mohan$^\dag$}}
	
\thanks{$\dag$ Department of Mathematics, Indian Institute of Technology Roorkee, Uttarakhand, 247667, India, Email- {\normalfont{ wakram2k11@gmail.com; maniltmohan@ma.iitr.ac.in}}\\
$\ddag$ Department of Mathematics, Indian Institute of Science, Education and Research, Thiruvananthapuram,  Kerala, 695551, India, Email- {\normalfont {manikabag19@iisertvm.ac.in}}\\
Dr. Wasim is supported by NBHM (National Board of Higher Mathematics, Department of Atomic Energy) postdoctoral fellowship, No. 0204/16(1)(2)/2024/R\&D-II/10823. Ms. Manika Bag would like to thank IIT Roorkee for providing stimulating scientific environment and resources.}

\makeatletter
\renewcommand{\tocsection}[3]{%
	\indentlabel{\@ifnotempty{#2}{\bfseries\ignorespaces#1 #2\quad}}\bfseries#3}
\renewcommand{\tocsubsection}[3]{%
	\indentlabel{\@ifnotempty{#2}{\ignorespaces#1 #2\quad}}#3}

\newcommand\@dotsep{4.5}
\def\@tocline#1#2#3#4#5#6#7{\relax
	\ifnum #1>\c@tocdepth 
	\else
	\par \addpenalty\@secpenalty\addvspace{#2}%
	\begingroup \hyphenpenalty\@M
	\@ifempty{#4}{%
		\@tempdima\csname r@tocindent\number#1\endcsname\relax
	}{%
		\@tempdima#4\relax
	}%
	\parindent\z@ \leftskip#3\relax \advance\leftskip\@tempdima\relax
	\rightskip\@pnumwidth plus1em \parfillskip-\@pnumwidth
	#5\leavevmode\hskip-\@tempdima{#6}\nobreak
	\leaders\hbox{$\m@th\mkern \@dotsep mu\hbox{.}\mkern \@dotsep mu$}\hfill
	\nobreak
	\hbox to\@pnumwidth{\@tocpagenum{\ifnum#1=1\bfseries\fi#7}}\par
	\nobreak
	\endgroup
	\fi}
\AtBeginDocument{%
	\expandafter\renewcommand\csname r@tocindent0\endcsname{0pt}
}
\def\l@subsection{\@tocline{2}{0pt}{2.5pc}{5pc}{}}
\makeatother

\begin{document}
	
	\pagenumbering{arabic}

	\begin{abstract}
   In this article, we investigate the stabilizability of the two- and three-dimensional  Navier-Stokes equations 
   with memory effects around a non-constant steady state using a localized interior control. The system is first linearized around a non-constant steady state and  then reformulated into a coupled system by introducing a new variable to handle the integral term. Due to the presence of variable coefficients in the linear operator, the rigorous computation of eigenvalues and eigenfunctions becomes infeasible. Therefore, we concentrate on the principal operator, and investigate its analyticity and spectral properties. We establish a feedback stabilization result for the principal system, ensuring a specific decay rate. Using the feedback operator derived from this analysis, we extend the approach to the full system, constructing a closed-loop system. By proving a suitable regularity result and applying a fixed-point argument, we ultimately demonstrate the stabilizability of the full system. We also discuss the stabilizability of the corresponding vorticity equation around a non-constant steady state. 
    \end{abstract}
\maketitle
\noindent \textbf{Keywords.} Navier-Stokes equations with memory; Feedback stabilization; Distributed control; Vorticity equation with memory.\\[2mm]
\noindent \textbf{MSC Classification (2020).} 35Q30 $\cdot$ 93D15 $\cdot$ 49J20 $\cdot$ 93C10


\section{Introduction}
The Navier-Stokes equations (NSEs) are fundamental to the study of fluid dynamics, describing the motion of fluids such as liquids, gases, and plasmas. NSEs are also pivotal in engineering applications, enabling the design of efficient systems such as aircraft, ships, pipelines, and chemical reactors. It is well recognized that in many physical systems, the current state of the fluid can depend on its past behavior. This motivates the study of the NSEs with memory (known as Oldroyd fluid flow equations in the literature), where the memory term or kernel accounts for the influence of past states on the current dynamics. The kernel often represents material properties such as viscoelastic effects, relaxation times, or other history-dependent characteristics of the fluid.

\subsection{Model problem and formulation}
The stabilization of Navier–Stokes flows highlights recent significant advancements in the mathematical theory of stabilizing Newtonian fluid flows. Finite-dimensional feedback controllers are employed to exponentially stabilize the equilibrium solutions of NSEs, effectively reducing or eliminating turbulence (see \cite{BarbuStabNSEBook} and references therein). In the present work, we aim to discuss the stabilizability of NSEs with kernel (non-Newtonian fluids) for the both velocity and vorticity equations by using finite/infinite feedback interior control.

Let $\Omega\subset\mathbb{R}^d,$ $d\in \{2,3\},$ be bounded domain with $C^2-$boundary $\Gamma.$ Let $\mathcal{O}\subset \Omega$ be an open subset and $\chi_{\mathcal{O}}$ denote the characteristic function. We aim to study feedback stabilizability of the following NSEs with memory:

\begin{equation} \label{eq:NSE_m_mod}
	\left\{
	\begin{aligned}
		& \yb_t -\eta \Delta \yb + (\yb\cdot \nabla) \yb +\nabla p - \kappa \int_0^t e^{-\lambda (t-s)}\Delta \yb(s) ds = \fb_\infty + \ub\chi_{\mathcal{O}} \text{ in } \Omega\times (0,\infty),\\
		& \nabla\cdot \yb=0 \text{ in } \Omega\times (0,\infty), \\
		& \yb=0  \text{ on } \Gamma\times (0,\infty), \\
		& \yb(0)=\yb_0 \text{ in }\Omega,
	\end{aligned}\right.
\end{equation}
where $\yb$ is the state variable (velocity field), $p$ is the pressure term and $\ub$ is the control acting in the interior $\mathcal{O}$ of the domain $\Omega.$ Here, $\eta, \kappa,$ and $\lambda$ are given positive constants, and $\fb_\infty \in (L^2(\Omega))^d$ is a given stationary force term. For the uniqueness of pressure, one may impose the condition $\int_{\Omega}p(x,t)dx=0$ for all $t\in(0,\infty)$.  In the first part of this article, our aim is to study the stabilizability of \eqref{eq:NSE_m_mod} around a non-constant steady state $\yb_{\infty}$  and then for the corresponding vorticity equation. In the literature, the model \eqref{eq:NSE_m_mod} is also known as equations of motion arising in Oldroyd fluids of order one (\cite{MTM20SAA,MTM00EECT}). The existence and uniqueness of a global weak solution of the above problem is discussed in the work \cite[Theorem 3.4]{MTM20SAA}.

Now, we formally derive the corresponding steady state equation. To obtain it, we first observe that
\begin{align*}
	\lim_{t\to \infty} \int_0^t e^{-\lambda (t-s)}\Delta \yb(s) ds = \lim_{t\to \infty}  \frac{\int_0^t e^{\lambda s}\Delta \yb(s) ds}{e^{\lambda t}} =  \lim_{t\to \infty} \frac{e^{\lambda t} \Delta \yb (t)}{\lambda e^{\lambda t}} = \frac{1}{\lambda} \Delta \yb_\infty,
\end{align*}
and 
\begin{align*}
	\frac{1}{\lambda} \Delta \yb_\infty =  
	\frac{1}{\lambda} (1-e^{-\lambda t}) \Delta \yb_\infty + \frac{1}{\lambda} e^{-\lambda t} \Delta \yb_\infty   =   \int_0^t  e^{-\lambda (t-s)} \Delta \yb_\infty ds+ \frac{1}{\lambda} e^{-\lambda t} \Delta \yb_\infty.
\end{align*}
Therefore, we consider the following  steady-state system: 
\begin{equation} \label{eq:StdNSE}
	\left\{
	\begin{aligned}
		&  - \left(\eta+\frac{\kappa}{\lambda}\right) \Delta \yb_\infty + (\yb_\infty \cdot \nabla) \yb_\infty +\nabla p_\infty  = \fb_\infty \text{ in } \Omega,\\
		& \nabla \cdot \yb_\infty =0 \text{ in } \Omega,  \\
		& \yb=0  \text{ on } \Gamma\times (0,\infty).
	\end{aligned}\right.
\end{equation}
The well-posedness of the system  \eqref{eq:StdNSE} is well-known and can be found in literature (e.g. see \cite[Theorem 10.1]{TemamNSE01},\cite{TemamSmBk}).  In fact,   for a given $\fb_\infty \in \Lb, $ for sufficiently  large $\eta+\frac{\kappa}{\lambda},$ there exists a unique strong solution $\yi \in \Hzz\cap \Vf$ satisfying 
\begin{align*}
	\|\yi\|_{\Hzz} \le C \|\fbi\|_{\Lb},
\end{align*}
for some $C>0$. Here, the spaces are defined as 
\begin{align*}
	& \Lb:=(L^2(\Omega))^d, \quad \mathbb{L}^p(\Omega)= (L^p(\Omega))^d, \quad 1\le p \le \infty,\\
	& \Hoz:= (H^1_0(\Omega))^d, \quad \Hb^s(\Omega):= (H^s(\Omega))^d, \quad s\in \mathbb{N},\\
	& \Hb:=\{ \boldsymbol{z} \in \Lb\, :\, \nabla \cdot \boldsymbol{z} =0, \, \boldsymbol{z}\cdot \boldsymbol{n}=0 \text{ on } \Gamma \}, \\
	& \Vf:= \{ \boldsymbol{z} \in \Hoz\,:\, \nabla \cdot \boldsymbol{z} =0 \},
\end{align*}
where $\boldsymbol{n}$ is the unit outward drawn normal to the boundary $\Gamma$ and $\boldsymbol{z}\cdot \boldsymbol{n}\big|_{\Gamma}$ is understood in the sense of trace (\cite[P. 15]{TemamSmBk},\cite{TemamNSE01}).  The norms in $\Hb$ and $\Vf$ are given by $\|\boldsymbol{z}\|_{\Hb}=\left(\int_{\Omega}|\boldsymbol{z}(x)|^2dx\right)^{1/2}$ and $\|\boldsymbol{z}\|_{\Vf}=\left(\int_{\Omega}|\nabla\boldsymbol{z}(x)|^2dx\right)^{1/2}$, respectively. 
To know the importance of the spaces defined above and for NSEs, one can see see the book \cite{TemamNSE01}. The system \eqref{eq:NSE_m_mod} is said to be exponentially stabilizable with decay $-\nu<0$ around the steady state pair $(\yb_\infty,p_\infty)$ of \eqref{eq:StdNSE} if there exists a control $\ub$ in $L^2(0,\infty; \Lb)$ satisfying 
\begin{align*}
	\|\ub \|_{L^2(0,\infty; \Lb)} \le C \|\yb_0 - \yb_\infty\|_{\Hb},
\end{align*}
for some $C>0$  such that the corresponding solution pair $(\yb,p)$ satisfies
\begin{align}\label{eqdef:stablzb}
	\|\yb(t)-\yb_\infty\|_{\Vf}+\|p(t)-p_\infty\|_{H^1(\Omega)} \le C_1(\|\yb_0 - \yb_\infty\|_{\Vf} ,\|\yb_\infty\|_{\Hzz}) e^{-\nu t}\text{ for all }t>0, 
\end{align}
and for some $C_1= C_1(\|\yb_0 - \yb_\infty\|_{\Vf},\|\yb_\infty\|_{\Hzz})>0.$

\subsection{Literature survey and motivations}
There are vast literature available on NSEs and on the feedback stabilizability, for references, one can see \cite{BreKun20,BreKunPf19,SChowErve,JPR6-2017,JPR15SICON,BBSW15,BarbSICON12,JPRSheThev,BarbSICON11,BadESAIM9,AKR24}, and references therein. However, we will mention a few of them that align in the present case. Exponential stabilization of the linearized Navier–Stokes equations around an unstable stationary solution by using boundary control is studied in \cite{JPR,JPR2010}. The authors in \cite{CRR-SICON,CMRR-JDE}  investigate stabilization of one-dimensional compressible NSEs around a constant steady state. Boundary feedback stabilization around an unstable stationary solution for a two-dimensional fluid flow governed by  NSEs with mixed boundary conditions can be found in \cite{JPR15SICON}. Local stabilization of one-dimensional compressible NSEs around a constant steady solution established in \cite{MRR_ADE}, while of three-dimensional viscous Burgers equation with memory is considered in \cite{WKR}. Based on the spectral properties of the Oseen–Stokes operator, \cite{MuntNSEMem} study NSEs with fading memory to study the stabilizability by boundary Dirichlet feedback controllers. In the later article, the authors studied the integral equations by writing it in a coupled form and then using semigroup theory. Here, the authors obtain a local stabilization result. For a non-linear parabolic system, using finite-dimensional feedback controller, stabilizability is demonstrated in \cite{Badra-T}. Breiten and Kunisch in \cite{BreKun14} focused on stabilizability by a Riccati based approach, and this approach is also extensively used in \cite{Las-Tr-91,Lasiecka1,Lasieckabook}. Munteanu, in \cite{Mont}, considered a semilinear heat equation with fading memory  to stabilize the system by using finite-dimensional boundary feedback controllers. A non-linear feedback controller by solving a Hamilton-Jacobi-Bellman and algebraic Riccati equation is constructed to stabilize two-dimensional Burgers equation by the authors in \cite{Thev-JPR}. 

Local stabilization of the viscous Burgers equation with memory is established in \cite{WKR} by reformulating the equation in a coupled form. Following their analysis, a similar approach could potentially be applied to NSEs with memory. However, their discussion is limited to stabilizability around the zero steady state. This raises the intriguing question of stabilizability around a non-constant steady state, which serves as the primary focus of this article. A key challenge arises due to the presence of variable coefficients in the operator of the coupled equations, which prevents the explicit determination of eigenvalues. Consequently, the existing analysis cannot be directly applied. We address this challenge in the next subsection, where we delve into the stabilizability around a non-constant steady state.

\subsection{Contributions, Methodology, and Novelty}
In this article, we focus on the stabilizability of velocity as well as vorticity equation by using localized interior controls and around a corresponding non-constant steady state. We first focus on the velocity equation.

To study the exponential stabilizability of $\yb$ around the non-constant steady state $\yb_\infty$, that is, a strong solution of the stationary NSEs given in \eqref{eq:StdNSE}, we first linearize \eqref{eq:NSE_m_mod} around $\yb_\infty$ and obtain a equation in $\zb=\yb-\yb_\infty$ and $q=p-p_\infty,$ which is also a system with memory:
\begin{equation} \label{eq:NSE_m_mod_z}
	\left\{
	\begin{aligned}
		& \zb_t -\eta \Delta \zb + (\zb\cdot \nabla) \zb + (\yb_\infty\cdot \nabla) \zb + (\zb\cdot \nabla) \yb_\infty +\nabla q  \\
		& \qquad \qquad - \kappa \int_0^t e^{-\lambda (t-s)}\Delta \zb(s) ds -\frac{\kappa}{\lambda} e^{-\lambda t}\Delta \yb_\infty=  \ub\chi_{\mathcal{O}} \text{ in } \Omega\times (0,\infty),\\
		& \nabla\cdot \zb=0 \text{ in } \Omega\times (0,\infty), \\
		& \zb=0  \text{ on } \Gamma\times (0,\infty), \\
		& \zb(0)=\yb_0-\yb_\infty=:\zb_0 \text{ in }\Omega.
	\end{aligned}\right.
\end{equation}
Therefore, we aim to find a feedback control such that the corresponding solution $\zb$ of \eqref{eq:NSE_m_mod_z} goes to zero exponentially in a suitable norm as $t\to \infty,$ in particular, to find a control so that $(\zb,q),$ the corresponding solution of \eqref{eq:NSE_m_mod_z} satisfies an analogous estimate as \eqref{eqdef:stablzb}, that is,
\begin{align*}
	\|\zb(t)\|_{\Vf} + \|q(t)\|_{H^1(\Omega)}\le C_1(\|\zb_0\|_{\Vf},\|\yb_\infty\|_{\Hzz}) e^{-\nu t}, \ \text{ for all }\ t>0,
\end{align*}
and for some $C_1= C_1(\|\zb_0\|_{\Vf},\|\yb_\infty\|_{\Hzz})>0.$
Here, we follow a similar technique as used in \cite{WKR} although they only discuss stabilizability around a zero steady state. We first write the integral equation as coupled system by introducing a new variable to the memory term as  $\wb(t):=\int_0^t e^{-\lambda (t-s)} \zb(s) ds,$ and on differentiating we obtain
\begin{align*}
	\wb_t+\lambda\wb-\zb=0.
\end{align*}
Thus we obtained a coupled system on $\Hf:=\Hb\times \Hb$ (equipped with the usual inner product and norm): 
\begin{equation} \label{eqn:non-linear intromain}
	\left\{
	\begin{aligned}
		& \zb_t - \eta \Delta \zb + (\zb\cdot\nabla) \zb +(\yb_\infty\cdot \nabla) \zb + (\zb\cdot \nabla) \yb_\infty \\
		& \qquad \qquad - \kappa \Delta \boldsymbol{w} -\frac{\kappa}{\lambda} e^{-\lambda t}\Delta \yb_\infty+\nabla p = \boldsymbol{u}\chi_{\mathcal{O}} \text{ in } \Omega\times (0,\infty), \\
		& \wb_t+\lambda\wb - \zb= 0 \text{ in } \Omega\times (0,\infty), \\
		&\nabla\cdot \zb=0, \quad\& \quad \nabla\cdot \boldsymbol{w}=0 \text{ in } \Omega\times (0,\infty), \\
		& \zb=0, \quad\& \quad \wb=0  \text{ on } \Gamma\times (0,\infty), \\
		& \zb(0)=\zb_0,  \quad\& \quad \wb(0)=0\text{ in }\Omega.
	\end{aligned}\right.
\end{equation}
Let $\Pb:\Lb \to \Hb$ be the Helmholtz-Hodge (or Leray) projection, and denote 
\begin{align}
	A_0:=-\Pb\Delta    \text{ with } D(A_0):=\Hzz \cap \Vf, 
\end{align}
as the Stokes operator (\cite{TemamNSE01}). 
Also  for $(\ub\cdot \nabla )\wb\in \Lb$, we denote   $B(\ub,\wb):=\Pb [(\ub\cdot \nabla )\wb].$ Then the above system can be written as
\begin{equation} \label{eq:NSENLinz}
	\left\{
	\begin{aligned}
		& \frac{d\zb}{dt}+\eta A_0\zb +B(\zb, \zb) +B(\yb_\infty, \zb) +B(\zb, \yb_\infty) \\
		& \qquad \qquad \qquad  + \kappa A_0 \wb + \frac{\kappa}{\lambda} e^{-\lambda t} A_0 \yb_\infty = \boldsymbol{v} \chi_O \text{ in } \Omega\times (0,\infty),\\
		& \frac{d\wb}{dt} +\lambda\wb - \zb =0 \text{ in } \Omega\times (0,\infty),\\
		& \zb =0 , \quad \& \quad \wb=0 \text{ on } \Gamma\times (0,\infty), \\
		& \zb(0)=\zb_0, \quad \& \quad \wb(0)=0 \text{ in }\Omega,
	\end{aligned}\right.
\end{equation} 
where $\boldsymbol{v}=\Pb \ub.$ To study the stabilizability with decay rate, our aim is to show that the corresponding unbounded operator generates an analytic semigroup and aim to study the spectral properties. However, if we write the above system as 
\begin{align*}
	\begin{pmatrix} \zb'(t) \\ \wb'(t) \end{pmatrix} =  \begin{pmatrix} \eta A_0 + A_1 & \kappa A_0 \\ I & -\lambda I \end{pmatrix} \begin{pmatrix} \zb(t) \\ \wb(t) \end{pmatrix} + \begin{pmatrix} B(\zb,\zb) + \frac{\kappa}{\lambda} e^{-\lambda t}A_0 \yb_\infty + \boldsymbol{v} \chi_{\mathcal{O}} \\ 0 \end{pmatrix},  
\end{align*}
with $A_1\zb= B(\yb_\infty, \zb) +B(\zb, \yb_\infty),$ then it is not easy to compute the eigenvalues of the operator and hence difficult to study spectral analysis. The method discussed in the book \cite[Section 3.3]{BarbuStabNSEBook} by V. Barbu, cannot be applied here as neither we can discuss the explicit behavior of the spectrum of $ \begin{pmatrix} \eta A_0 + A_1 & \kappa A_0 \\ I & -\lambda I \end{pmatrix} $ nor the operator $ \begin{pmatrix} \eta A_0 + A_1 & \kappa A_0 \\ I & -\lambda I \end{pmatrix} $ has compact resolvent. Although one could directly examine the linear part of the full system for stabilizability, the presence of variable coefficients in the operator significantly complicates the computation of its eigenvalues and eigenfunctions, making spectral analysis challenging. As a result, our approach may not be feasible in that case. To overcome this and manage this difficulty, instead of taking the whole linear operator, we focus only on the principal part, that is,
\begin{equation} \label{w4}
	\left\{
	\begin{aligned}
		& \frac{d\zb}{dt}+\eta A_0\zb   + \kappa A_0 \wb  = \boldsymbol{v} \chi_O \text{ in } \Omega\times (0,\infty),\\
		& \frac{d\wb}{dt} +\lambda\wb - \zb =0 \text{ in } \Omega\times (0,\infty),\\
		& \zb =0 , \quad \& \quad \wb=0 \text{ on } \Gamma\times (0,\infty), \\
		& \zb(0)=\zb_0, \quad \wb(0)=0 \text{ in }\Omega.
	\end{aligned}\right.
\end{equation} 
Therefore, in the next section, our focus will be on stabilizability of the system \eqref{w4} and to obtain a corresponding feedback control. The feedback operator obtained in this setting can then be utilized to analyze the stabilizability of the full system. Considering this type of principal system is new to the literature according to our knowledge and this idea helps to study the stabilizability of  evolution equations with memory around a non-constant steady state. By denoting  $Y(t):=\left( \begin{matrix}
	\zb(\cdot,t)\\ \wb(\cdot,t)
\end{matrix} \right)$, \eqref{w4} can be written as 
\begin{equation} \label{eqn:main_system}
	Y'(t)=\Ac Y(t)+\Bc \boldsymbol{v}(t) \, \text{ for all }\, t>0, \quad  Y(0)=\left( \begin{matrix}
		\zb_0\\ 0
	\end{matrix} \right):=Y_0,
\end{equation} 
for $Y_0\in \Hf$ and $\boldsymbol{v}\in L^2(0,\infty; \Hb)$, where the unbounded operator $(\Ac, D(\Ac))$ on $\Hf$ and the control operator $\Bc:\Hb\to\Hf$ are defined as
\begin{equation} \label{123}
	\Ac:=\left( \begin{matrix}
		-\eta A_0 & -\kappa A_0 \\ I & -\lambda I 
	\end{matrix} \right)\text{ with } D(\Ac)=\left\lbrace \begin{pmatrix} \zb \\ \wb \end{pmatrix}\in \Hf :\eta \zb+\kappa\wb \in D(A_0) \right\rbrace,
\end{equation}
and
\begin{equation}\label{eqcontrol} 
	\Bc\boldsymbol{v}:=\left( \begin{matrix}
		\boldsymbol{v}\chi_{\mathcal{O}}\\0
	\end{matrix} \right) \text{ for all }\boldsymbol{v}\in \Hb.
\end{equation}
Now, observe that the one can easily study the system \eqref{eqn:main_system} as the unbounded operator $(\Ac, D(\Ac))$ has constant coefficients and hence one can study the spectral properties as the eigenvalue problems for the Stokes operator $A_0$ is known to us. Thus, we focus on to study the stabilizability of the principal system \eqref{w4} or \eqref{eqn:main_system} with decay rate by using a Hautus type of Criterion (see Proposition \ref{prop:stab by Hautus}). Here, we obtain the decay rate $-\nu<0$ for any $\nu\in (0,\nu_0),$ where $\nu_0$ is an accumulation point of the eigenvalues of $\Ac$ (see Theorem \ref{th:sspecAna} and \eqref{eqlambda0}). Now, one can easily construct a (finite- or infinite-dimensional) feedback control by solving an optimal control problem and (degenerate/non-degenerate) algebraic Riccati equation as studied in Section 4 of \cite{WKR}. A non-linear feedback control by using a power series expansion method, which is a formal approximate solution to a Hamilton-Jacobi-Bellman equation, is obtained in \cite{Thev-JPR} to study feedback stabilizability of a two dimensional Burgers equation. We aim to find such a non-linear feedback control for the present model in a forthcoming paper.

Once, a feedback stabilizing control (in Theorem \ref{th:stb cnt_vel}) is obtained for the principal system, we plug it into the original system in coupled form, then studying regularity results and using Banach fixed point theorem, we establish that the closed loop system is exponentially stable (see Theorems \ref{th:vstab non lin} - \ref{th:mainresultStabIntegral}).

In the next part of this article, we move on to study the stabilizability of the corresponding vorticity equation obtained by taking `curl' in the model problem \eqref{eq:NSE_m_mod} around a non-constant steady state. Here the equation is obtained in a scalar form. Here also, we first consider the principal system of the linearized equation linearized around the non-cconstant steady state and show that the operator generates an analytic semigroup and study the spectral analysis. Then, the remaining analysis to obtain the stabilizability (see Theorems \ref{th:stab non lin} - \ref{th:stab non linIntgrl}) are similar to the analysis done in the case of velocity equation, and hence we don't discuss the details here.

The novelty of this article is that the feedback stabilization of a non-linear integro-differential equation around a non-constant steady state is new. Also, to the best of our knowledge,  the technique of considering the principal system where the spectral analysis can be done thoroughly is new to the literature. This technique can be applied to any evolution equation with memory of the type considered in this work, to study the stabilizability or controllability around a non-constant steady state where the semigroup and spectral properties of the principal operator can be studied. We also discuss a local well-posedness result of three-dimensional NSEs (with/without memory) in Remark \ref{rem:Solv3DNSE2}.


\subsection{Organization} 
The rest of the article is organized as follows. We discuss the semigroup and spectral properties of the principal operator $\Ac$ in Section \ref{sec:PS-NSE}. Subsequently, here, we also discuss the feedback stabilizability of the principal system. Using the feedback operator obtained in the stabilizability of the principal system, some regularity results for non-linear system, and the Banach fixed point theorem, Section \ref{sec:FS-NSE} concludes the local stabilization of NSEs with memory around a non-constant steady state. The stabilizability of vorticity equation around a non-constant steady state is considered in Appendix \ref{app:Vort}.

\section{Principal System} \label{sec:PS-NSE}
This section is devoted to study the stabilizability of the principal system \eqref{w4} and we obtain a feedback operator that will be used in the full system to show the stabilizability of the main system.

For the purpose of our analysis, recall the operator form \eqref{eqn:main_system} of \eqref{w4}, and the operators $(\Ac, D(\Ac))$ and $\Bc \in \mathcal{L}(\Hb,\Hf)$. The adjoint operator of $(\Ac, D(\Ac))$ on $\Hf$ is defined by $(\Ac^*, D(\Ac^*))$, where 
\begin{equation}\label{eqadjopt}
	\Ac^*=\left( \begin{matrix}
		-\eta A_0 & I \\ -\kappa A_0 & -\lambda I 
	\end{matrix} \right) \text{ and } D(\Ac^*)= \left\lbrace \begin{pmatrix} \zb \\ \wb \end{pmatrix} \in \Hf : \zb\in D(A_0) \right\rbrace,
\end{equation}
and the adjoint operator of $\Bc$, denoted by $\Bc^*$, belongs to $\mathcal{L}(\Hf,\Hb),$ and $\Bc^*$ is defined by 
\begin{equation}\label{eqadjcontrol}
	\Bc^*\left( \begin{matrix}
		\boldsymbol{\phi} \\ \boldsymbol{\psi} \end{matrix} \right)= \boldsymbol{\phi} \chi_\mathcal{O} \text{ for all } \left( \begin{matrix}
		\boldsymbol{\phi} \\ \boldsymbol{\psi} \end{matrix} \right)\in \Hf.
\end{equation}

\subsection{Analytic Semigroup and Spectral Analysis} \label{sec:spectral analysis}
This section focuses on the unbounded operator $(\Ac, D(\Ac))$ to show that it generates an analytic semigroup on $\Hf$ and that yields the well-posedness of the system  
\begin{equation}\label{eqwell1}
	Y'(t)=\Ac Y(t)+F(t) \, \text{ for all }\, t>0, \quad Y(0)=Y_0,
\end{equation} 
for any $Y_0\in \Hf$ and $F\in L^2(0,\infty; \Hf)$.
and then discusses the spectral analysis. These properties of $\Ac$ are crucial to discuss the stabilizability of \eqref{eqn:main_system}.

\begin{Theorem} \label{pps:existence}
	(a) The operator $(\Ac,D(\Ac))$ in \eqref{123} is densely defined and closed. The half plane $\lbrace \mu \in \mathbb{C} \mid \Re \mu>0 \rbrace$ is contained in the resolvent set of $\Ac$ and
	$$\| (\mu I-\Ac)^{-1}\|_{\mathcal{L}(\Hf)} \leq \frac{C}{\vert \mu\vert}, \text{ for all } \mu\in \C, \text{ with }  \Re(\mu)>0,$$
	for some $C>0$ independent of $\mu.$
	\noindent Consequently, $(\Ac,D(\Ac))$ generates an analytic semigroup $\{e^{t\Ac}\}_{t\ge 0}$ on $\Hf.$\\ 
	(b) Furthermore, for any $Y_0\in \Hf$ and any $F\in L^2(0,\infty; \Hf)$,  system  \eqref{eqwell1}
	admits  a unique solution $Y \in L^2(0,\infty; \Hf).$ Moreover, the solution $Y$ belongs to $ C([0,\infty); \Hf)$ with the representation 
	\begin{equation}\label{eqrep1} 
		Y(t)=e^{t\Ac}Y_0+\int_0^te^{(t-s)\Ac}F(s)ds, \text{ for all } t>0.
	\end{equation}
\end{Theorem}
\begin{proof}
	The proof can be established by mimicking the proof of  \cite[Proposition 3.1 and Theorem 3.2]{WKR}, and then (b) follows from \cite[Proposition 3.1, Chapter 1, Part II]{BDDM}.
\end{proof}

Now, we discuss the behavior of the spectrum of $\Ac$ on $\Hf$.  
We recall the result (for details, see \cite[Theorem 2.24]{RRS16}) that 
there exists an orthonormal basis $\lbrace\Phi_m\rbrace_{m\in  \mathbb{N}}$ of $\Hb$ and sequence of positive real numbers $\lbrace \sigma_m\rbrace_{m \in \mathbb{N}}$ with $\sigma_m\rightarrow \infty$ as $m\rightarrow \infty,$ such that 
\begin{equation}\label{eqevlaplace}
	\left\lbrace
	\begin{aligned}
		& 0<\sigma_1\leq \sigma_2\leq ...\leq \sigma_m\leq ... \nearrow \infty, \\
		& A_0 \Phi_m =\sigma_m \Phi_m,  \text{ in }\Omega,  \\
		& \Phi_m \;\in D(A_0)\cap C^\infty(\overline{\Omega}).
	\end{aligned}\right.
\end{equation}
\begin{Theorem}[spectral analysis]\label{th:sspecAna}
	Let $(\Ac, D(\Ac))$ be as defined in \eqref{123}. The eigenvalues of $\Ac$ are consists of two sequences $\mu_k^+$ and $\mu_k^-,$ where
	\begin{align}
		\mu_k^\pm = \frac{-(\lambda+\eta \sigma_k) \pm \sqrt{(\lambda+\eta \sigma_k)^2 - 4\sigma_k (\eta \lambda+\kappa)}}{2},
	\end{align}
	where $\sigma_k$ are eigenvalues of $A_0$ as discussed in \eqref{eqevlaplace}. Furthermore, we have the following properties: 
	\begin{itemize}
		\item[(a)] There are only finitely many complex eigenvalues of $\Ac.$ In fact, $\mu_k^\pm,$ are complex if and only if correspomding $\sigma_k$ satisfies 
		\begin{align*}
			\frac{(\eta\lambda+2 \kappa) - 2 \sqrt{\kappa^2+\kappa \eta \lambda}}{\eta^2} <\sigma_k < \frac{(\eta\lambda+2 \kappa) + 2 \sqrt{\kappa^2+\kappa \eta \delta}}{\lambda^2}.
		\end{align*}
		\item[(b)] The sequence $\mu_k^+$ converges to $-\nu_0$ as $k\rightarrow \infty$ while the other sequence $\mu_k^-$ behaves like $-\eta\sigma_k$ and goes to $-\infty$ as $k\to \infty,$ where
		\begin{align} \label{eqlambda0}
			\nu_0:= \left(\frac{\kappa}{\eta}+\lambda\right).  
		\end{align}
		\item[(c)] All the eigenvalues of $\Ac$ have negative real part. In particular, there exists $N_1,N_2\in \mathbb{N}$ such that 
		\begin{align*}
			& \text{ for all }k< N_1, \quad  -\nu_0 < \Re(\mu_k^+)<0 ,    \quad \text{ for all }k\ge N_1, \, \Re(\mu_k^+) \le -\nu_0, \\
			& \text{ for all }k< N_2, \quad  -\nu_0 < \Re(\mu_k^-)<0 ,    \quad \text{ for all }k\ge N_2, \, \Re(\mu_k^-) \le -\nu_0.
		\end{align*}
	\end{itemize}
\end{Theorem}

Since, $\{\Phi_k\}_{k\in \mathbb{N}}$ is an orthonormal basis of $\Hb,$ one can easily show that $\left\lbrace \begin{pmatrix} \Phi_k \\0 \end{pmatrix}, \begin{pmatrix} 0 \\ \Phi_k \end{pmatrix}, \, k\in \mathbb{N} \right\rbrace $ is an orthonormal basis of $\Hf.$ Now, we choose eigenfunctions $\boldsymbol{\zeta}_k^\pm$ of $\Ac$ corresponding to eigenvalues $\mu_k^\pm$ as 
\begin{align} \label{eq:seigfun-A}
	\boldsymbol{\zeta}_k^\pm= \begin{pmatrix} 1 \\ \frac{1}{\mu_k^\pm+\lambda} \end{pmatrix}\Phi_k \text{ for all }k\in \mathbb{N}.
\end{align}
Observe that the set of eigenvalues of $\Ac^*$ is $\left\lbrace \overline{\mu_k^+}, \overline{\mu_k^-}, \, |\, k \in \mathbb{N} \right\rbrace$. For all $k\in \mathbb{N},$ we calculate eigenfunctions $\boldsymbol{\zeta}_k^{*+}$ and $\boldsymbol{\zeta}_k^{*-}$ of $\Ac^*$ corresponding to eigenvalues $\overline{\mu_k^+}$ and $ \overline{\mu_k^-},$ respectively as
\begin{align} \label{eq:seigfun-A*}
	\boldsymbol{\zeta}_k^{*+} =    \frac{(\lambda+ \overline{\mu_k^+})^2}{(\lambda+ \overline{\mu_k^+})^2-\kappa \sigma_k}  \begin{pmatrix} 1 \\ \frac{-\kappa \sigma_k}{\lambda+ \overline{\mu_k^+}} \end{pmatrix} \Phi_k \text{ and } \boldsymbol{\zeta}_k^{*-} =    \frac{(\lambda+ \overline{\mu_k^-})^2}{(\lambda+ \overline{\mu_k^-})^2-\kappa \sigma_k}  \begin{pmatrix} 1 \\ \frac{-\kappa \sigma_k}{\lambda+ \overline{\mu_k^-}} \end{pmatrix} \Phi_k.
\end{align}

This set of families of eigenfunctions of $\Ac$ and $\Ac^*$ satisfies the bi-orthonormality relations.

\begin{Proposition}[bi-orthonormality]\label{pps:NSEmBiorth}
	For all $k\in \mathbb{N},$ let $\{ \boldsymbol{\zeta}_k^{+} , \boldsymbol{\zeta}_k^{-}\} $ and $\{\boldsymbol{\zeta}_k^{*+}, \boldsymbol{\zeta}_k^{*-}\}$ be sets of eigenfunctions of $\Ac$ and $\Ac^*,$ respectively, as defined in \eqref{eq:seigfun-A} and \eqref{eq:seigfun-A*}, respectively. The the following bi-orthonormality relations hold:
	\begin{align*}
		& \left( \boldsymbol{\zeta}_k^+, \boldsymbol{\zeta}_n^{*+}\right)_{\Hf}  = \delta_k^n, \qquad \left( \boldsymbol{\zeta}_k^+, \boldsymbol{\zeta}_n^{*-}\right)_{\Hf} =0 \text{ for all }k, n \in \mathbb{N}, \\
		& \left( \boldsymbol{\zeta}_k^-, \boldsymbol{\zeta}_n^{*-}\right)_{\Hf}  = \delta_k^n, \qquad \left( \boldsymbol{\zeta}_k^-, \boldsymbol{\zeta}_n^{*+}\right)_{\Hf} =0 \text{ for all }k, n \in \mathbb{N},
	\end{align*}
	where 
	\begin{align*}
		\delta_k^n =\begin{cases} 1 & \text{ if } k=n, \\
			0 & \text{ if } k\neq n.\end{cases}
	\end{align*}
\end{Proposition}
Recall the eigenfunctions of $\Ac$ and $\Ac^*$ given in \eqref{eq:seigfun-A} and \eqref{eq:seigfun-A*}, respectively. Following the same argument and a similar calculation of \cite[Proposition 3.8]{WKR}, one can uniquely represents any $\boldsymbol{\zeta}$ in $\Hf$ as 
\begin{equation} \label{eqn:full_expression}
	\boldsymbol{\zeta}=\sum_{n=1}^\infty  \left( \zeta,\zeta^{*+}_n \right)_\Hf\zeta_n^+ +\sum_{n=1}^\infty \left( \zeta,\zeta^{*-}_n \right)_\Hf\zeta_n^-.
\end{equation}
There exist positive $C_1$ and $C_2$ such that 
{\small
	\begin{equation}\label{eqnorm}
		C_1\sum_{n=1}^{\infty}\left(\left\vert  \left( \boldsymbol{\zeta},\boldsymbol{\zeta}^{*+}_n \right)_\Hf\right \vert^2+\left\vert \left(  \boldsymbol{\zeta},\boldsymbol{\zeta}^{*-}_n \right)_\Hf\right\vert^2\right)\leq \|\boldsymbol{\zeta}\|_\Hf^2\leq C_2 \sum_{n=1}^{\infty}\left(\left\vert  \left( \boldsymbol{\zeta},\boldsymbol{\zeta}^{*+}_n \right)_\Hf\right\vert^2+ \left\vert \left( \boldsymbol{\zeta},\boldsymbol{\zeta}^{*-}_n \right)_\Hf\right \vert^2\right).
\end{equation}  }
Also, $\boldsymbol{\zeta}  \in \Hf$ can be uniquely represented by using eigenfunctions of $\Ac^*$, $\{\boldsymbol{\zeta}^{*+}_k, \boldsymbol{\zeta}^{*-}_k\mid k\in \boldsymbol{N}\}$ as 
\begin{equation} \label{eqn:full_expressionadj}
	\boldsymbol{\zeta}=\sum_{n=1}^\infty  \left( \boldsymbol{\zeta},\boldsymbol{\zeta}^+_n \right)_\Hf \boldsymbol{\zeta}_n^{*+} +\sum_{n=1}^\infty \left( \boldsymbol{\zeta},\boldsymbol{\zeta}^-_n \right)_\Hf\boldsymbol{\zeta}_n^{*-},
\end{equation}
and for some positive constants $C_1^*$ and $C_2^*,$ we have
\begin{equation}\label{eqnormadj}
	C_1^*\sum_{n=1}^{\infty}\left(\left\vert \left( \boldsymbol{\zeta},\boldsymbol{\zeta}^+_n \right)_\Hf\right \vert^2+ \left\vert \left( \boldsymbol{\zeta},\boldsymbol{\zeta}^-_n \right)_\Hf\right\vert^2\right)\leq \|\boldsymbol{\zeta}\|_\Hf^2\leq C_2^* \sum_{n=1}^{\infty}\left(\left\vert \left( \boldsymbol{\zeta},\boldsymbol{\zeta}^+_n \right)_\Hf \right\vert^2+ \left\vert \left( \boldsymbol{\zeta},\boldsymbol{\zeta}^-_n \right)_\Hf \right \vert^2\right).
\end{equation}  
As a consequence of the above proposition, from \cite[Remark 2.6.4, Chapter 2, Section 2.6]{Tucs}, we get the following result of spectrum of $\Ac$ and $\Ac^*$. 
\begin{Theorem}\label{thmspectrum}
	The family of eigenfunctions of $\Ac$, $\{\zeta^+_k, \zeta^-_k\mid k\in \mathbb{N}\}$, forms a Riesz basis in $\Hf$.  The spectrum of $A,$ denoted by $\sigma(\Ac),$ is 
	$$ \sigma(\Ac)= \mbox{the closure of}\quad \{\mu^+_k, \mu^-_k\mid k\in \mathbb{N}\} \quad \mathrm{in}\quad \mathbb{C} .$$
	In particular, $\sigma(\Ac)=\{\mu^+_k, \mu^-_k\mid k\in \mathbb{N}\}\cup\{-\nu_0\}$, where $\nu_0$ is defined in \eqref{eqlambda0}. \\
	Similarly, $\sigma(\Ac^*)$, the spectrum of $\Ac^*$, is  $\{\overline{\mu}^+_k, \overline{\mu}^-_k\}_{k\in \mathbb{N}}\cup \{-\nu_0\}$ in $\C$. 
\end{Theorem}

\begin{Remark}
	Note that the operator $\Ac$ on $\Hf$ can exhibit multiple eigenvalues, which may arise either due to the multiplicity of the eigenvalues of the Stokes operator $A_0$ on $\Hb$, or when, for some $k \in \mathbb{N}$, $\mu_k^+ = \mu_k^-$. In the first case, the eigenfunctions $\left\lbrace \boldsymbol{\zeta}_k^\pm \mid k \in \mathbb{N} \right\rbrace$ form a basis for $\Hf$, similar to how $\left\lbrace\Phi_k \mid k \in \mathbb{N}\right\rbrace$ forms a basis for $\Hb$, and the analysis proceeds naturally. In the latter case, where multiple eigenvalues occur, we compute the generalized eigenfunctions and proceed accordingly. For further details and procedures, we refer to Remarks 3.5 and 4.6 in \cite{WKR}, where an analogous formulation leads to the desired result.
\end{Remark}

\subsection{Stabilizability} This section is devoted not just to study the stabilizability of \eqref{eqn:main_system} but also to discuss the stabilizability with the expected decay rate. From the analysis of spectrum discussed in the previous section, one can expect that the decay rate of the system is bounded by $\nu_0,$ that is, we expect the decay rate as $\nu$ for any $\nu \in (0,\nu_0).$ To get the stabilizability with decay $\nu\in (0, \nu_0),$ it is convenient to discuss the stabilizability of the shifted system
\begin{equation} \label{eqn: main shiftedVel}
	\widetilde{Y}'(t)=\Ac_\nu \widetilde{Y}(t)+\Bc \widetilde{\boldsymbol{v}}(t), \, \text{ for all } \, t>0, \quad  \widetilde{Y}(0)=\left( \begin{matrix}
		\zb_0\\ 0
	\end{matrix} \right):= {Y}_0,
\end{equation}
where $\widetilde{Y}(t)=e^{\nu t}Y(t)$ and $\widetilde{\boldsymbol{v}}(t)=e^{\nu t}\boldsymbol{v}(t),$ and
\begin{align}\label{eqdef:Ac_nu}
	\Ac_\nu=\Ac+\nu \boldsymbol{I} \text{ with } D(\Ac_\nu)=D(\Ac),
\end{align}
where $\boldsymbol{I}:\Hf \to \Hf$ being identity operator. Observe that $Y$ and $\boldsymbol{v}$ can be recovered by
$ Y(t)= e^{-\nu t}\widetilde{Y}(t), \quad \boldsymbol{v}(t)= e^{-\nu t} \widetilde{\boldsymbol{v}}(t), \text{ for all } t>0,$ and satisfy \eqref{eqn:main_system}. Moreover, if \eqref{eqn: main shiftedVel} is exponentially stabilizable by control $\widetilde{\boldsymbol{v}},$ that is, 
for some positive constants $C$ and $\gamma$, 
$$ \|\widetilde{Y}(t)\|_\Hf \le C e^{-\gamma t} \|Y_0\|_\Hf,  \quad \|\widetilde{\boldsymbol{v}}(t)\|_{\Lb}\le C e^{-\gamma t} \|Y_0\|_\Hf, \text{ for all } t>0, $$
then $Y(\cdot)$ and $\boldsymbol{v}(\cdot)$ also obey
$$ \|Y(t)\|_\Hf \le C e^{-(\nu + \gamma )t} \|Y_0\|_\Hf, \quad \|\boldsymbol{v}(t)\|_{\Lb}\le C e^{-(\nu +\gamma) t} \|Y_0\|_\Hf, \text{ for all } t>0.$$
Thus, for any given $\nu\in (0,\nu_0)$, the feedback exponential stabilizability of \eqref{eqn: main shiftedVel} in $\Hf$  yields the exponential stabilizability of \eqref{eqn:main_system} in $\Hf$ with decay $-\nu<0$. Hence, we focus on the study of stabilizability of \eqref{eqn: main shiftedVel}.

Theorem \ref{thmspectrum} yields the spectrum of $\Ac_\nu$ to be $$\sigma(\Ac_\nu)=\{\mu^+_k+\nu, \mu^-_k+\nu\mid k\in \mathbb{N}\}\cup\{\nu-\nu_0 \},$$ where $\{\mu^+_k, \mu^-_k\mid k\in \mathbb{N}\}$ is the set of eigenvalues of $\Ac$. Note that from Theorem \ref{th:sspecAna},  for any $\nu\in (0,\nu_0)$, there exists only finitely many eigenvalues of $\Ac$ with real part greater than $-\nu$. In view of Theorem $\ref{th:sspecAna}$, without loss of generality, we can choose 
$\nu\in (0,\nu_0)$ such that $-\nu$ does not coincide with any eigenvalues of $\Ac$ and 
there exists $N_\nu\in \mathbb{N}$ satisfying
\begin{equation}\label{eqevorder}
	\begin{array}{l}
		\text{ for all } 1\le k \le N_\nu, \quad  -\nu<\Re(\mu^\pm_k), \quad 
		\text{ and for all }\, k>N_\nu, \quad \Re(\mu^\pm_k)<-\nu. 
	\end{array}
\end{equation}
Therefore there are finite number of eigenvalues of $\Ac_\nu$ which has positive real part, in particular,
$$
\begin{array}{l}
	\Re\,(\mu^\pm_k+\nu )>0, \quad \text{ for all }\, 1\le k\le N_\nu, \quad \text{ and } \,
	\Re\,(\mu^\pm_k+\nu )<0, \, \text{ for all }\, k>N_\nu.
\end{array}
$$
We denote the set of elements in the spectrum of $\Ac_\nu$ with positive real part by 
\begin{equation}\label{eqev+} 
	\sigma_+(\Ac_\nu)=\{\mu^+_k+\nu, \quad \mu^-_k+\nu\mid 1\le k\le N_\nu\},
\end{equation}
and the set of elements in the spectrum of $\Ac_\nu$ with negative real part by $\sigma_-(\Ac_\nu)$. In this case,  we have
\begin{equation}\label{eqev-}
	\sigma_-(\Ac_\nu)= \sigma(\Ac_\nu)\backslash \sigma_+(\Ac_\nu).
\end{equation}
The finite set $\sigma_+(\Ac_\nu)$  in $\mathbb{C}$ can be enclosed by a simple closed Jordan curve $\Gamma_u$. Then we set the projection $\pi_u\in \mathcal{L}(\Hf)$ associated  to $\sigma_+(\Ac_\nu)$ by 
\begin{align*}
	\pi_u :=\frac{1}{2\pi i}\int_{\Gamma_u} R(\mu,\Ac_\nu) \, d\mu,
\end{align*}
and $\pi_s:=(I- \pi_u)\in \mathcal{L}(\Hf)$ associated to $\sigma_-(\Ac_\nu)$.

For any $\nu\in (0,\nu_0)$, in the next proposition, the stabilizability of $(\Ac_\nu, \Bc)$ is proved by showing that Hautus condition holds. 
\begin{Proposition} \label{prop:stab by Hautus}
	For all $\nu\in (0,\nu_0)$, where $\nu_0$ is defined in \eqref{eqlambda0}, the system \eqref{eqn: main shiftedVel} is stabilizable in $\Hf$. 
\end{Proposition}
\begin{proof}
	We prove this result using \cite[Proposition 3.3, Chapter 1, Part-V]{BDDM}. Then the  following holds:
	\begin{itemize}
		\item[(a)] Since $(\Ac,D(\Ac))$ generates an analytic semigroup on $\Hf$ and $(\Ac_\nu,D(\Ac_\nu))$ is a bounded perturbation of $(\Ac,D(\Ac))$. Hence, by \cite[Theorem 12.37]{RROG}, $(\Ac_\nu,D(\Ac_\nu))$ generates an analytics semigroup on $\Hf.$
		\item[(b)] From \eqref{eqev+}, it follows that the set of eigenvalues of $\Ac_\nu$ with non-negative real part is finite.
		\item[(c)] Also, from \eqref{eqev-} and using a similar argument as \cite[Lemma 3.11]{WKRPCE}, there exist $\epsilon>0$ and constant $C>0,$ such that
		\begin{align*}
			\sup_{\mu \in \sigma_-(\Ac_\nu)} \Re(\mu)<-\epsilon \text{ and } \|e^{t\Ac_\nu}\pi_s\|_{\mathcal{L}(\Hf)}\le C e^{-\epsilon t} \text{ for all }t>0.
		\end{align*}
	\end{itemize}
	It remains to show that Hautus conditions satisfied by $\Ac_\nu^*$, indeed  
	\begin{align*}
		\text{Ker}(\mu I- \Ac_\nu^*)\cap \text{Ker} (\Bc^*) =\{0\}
	\end{align*}
	holds.
	Let $\boldsymbol{\zeta} \in \text{Ker}(\mu I- \Ac_\nu^*)\cap \text{Ker} (\Bc^*)$ be arbitrary. Now, $\boldsymbol{\zeta} \in \text{Ker}(\mu I- \Ac_\nu^*)$ implies $\boldsymbol{\zeta} $ is an eigenfunction of $\Ac_\nu^*$ and hence an eigenfunction of $\Ac^*.$ Therefore, $\boldsymbol{\zeta} =C_k\boldsymbol{\zeta}_k^{*+}$ or $\boldsymbol{\zeta} =C_k\boldsymbol{\zeta}_k^{*-}$, for some $k\in\{1,2, \dots, N_\nu\}.$ Furthermore, $\boldsymbol{\zeta} \in \text{Ker} (\Bc^*)$ implies $C_k\Phi_k\chi_\mathcal{O}=0,$ for all $k\in \{1,2, \dots, N_\nu\}$ that is, $C_k\Phi_k \mid_{\mathcal{O}}=0$ for all $k\in \{1,2, \dots, N_\nu\}$. But, since  $ \Phi_k$ is an eigenfunction of $A_0$ corresponding to the eigenvalue $\sigma_k$ in $\mathbb{H}$, $\Phi_k$ is analytic function in $\mathcal{O}$, an open connected domain of $\Omega$ for every $k.$ Therefore it cannot vanish identically in $\mathcal{O},$ and hence $C_k=0,$ for all $k\in \{1,2, \dots, N_\nu\}$ that is, $\boldsymbol{\zeta} =0.$ Therefore, 
	\begin{align*}
		\text{Ker}(\mu I- \Ac_\nu^*)\cap \text{Ker} (\Bc^*) =\{0\},
	\end{align*}
	and hence by \cite[Proposition 3.3, Chapter 1, Part-V]{BDDM}, the pair $(\Ac_\nu, \Bc)$ is open loop stabilizable in $\Hf.$
\end{proof}

Often, the feedback operator $K$ is obtained by studying an optimization problem and by using a Riccati equation. 
Consider the optimal control problem:
\begin{align}\label{eqoptinf_vel}
	\min_{\wt{\boldsymbol{v}}\in E_{Y_0}} J( \wt{Y},\wt{\boldsymbol{v}}) \text{ subject to \eqref{eqn: main shiftedVel}}, \text{ where } J( \wt{Y},\wt{\boldsymbol{v}}):=\int_0^\infty\big( \|\wt{Y}(t)\|_{\Hf}^2 + \|\wt{\boldsymbol{v}}(t)\|_\Hb^2\big) \, dt,
\end{align}
$E_{Y_0}:=\{ \wt{\boldsymbol{v}}\in L^2(0,\infty; \Hb)\mid \wt{Y} \text{ solution of \eqref{eqn: main shiftedVel} with control }\wt{\boldsymbol{v}} \text{ such that } J( \wt{Y},\wt{\boldsymbol{v}})<\infty\}.$
The next theorem yields the minimizer of \eqref{eqoptinf_vel} as well as the stabilizing control in the feedback form.  

\begin{Theorem}[{\normalfont \cite[Theorem 2.2 and Remark 4.5]{WKR}}]\label{th:stb cnt_vel}
	Let $\nu \in (0, \nu_0)$ be  any real number. Let $\Ac_\nu$ (resp. $\Bc$) be as defined in \eqref{eqdef:Ac_nu} (resp. \eqref{eqcontrol}). Then the following results hold: 
	\begin{enumerate}
		\item[(a)] There exists a unique operator $\Pc\in \mathcal{L}(\Hf)$ that satisfies the non-degenerate Riccati equation 
		\begin{equation}\label{eqn:ARE_vel}
			\begin{array}{l}
				\normalfont\Ac_\nu^*\Pc+\Pc\Ac_\nu-\Pc\Bc\Bc^*\Pc+\boldsymbol{I}=0,\quad \Pc=\Pc^* \geq 0 \text{ on }\Hf.
			\end{array}
		\end{equation}
		\item[(b)] For any $Y_0\in \Hf$, there exists a unique optimal pair $\normalfont(Y^\sharp,\boldsymbol{v}^\sharp)$ for \eqref{eqoptinf_vel}, where for all $t>0$, $Y^\sharp(t)$ satisfies the closed loop system 
		\begin{equation}\label{eqcl-loop_vel}
			Y{^\sharp}'(t)=(\Ac_\nu-\Bc\Bc^*\Pc)Y^\sharp(t),\;\; Y^\sharp(0)=Y_0,
		\end{equation}
		$\boldsymbol{v}^\sharp(t)$ can be expressed in the feedback form as
		\begin{equation}\label{eqoptcntrl_vel}
			\boldsymbol{v}^\sharp(t)=-\Bc^*\Pc Y^\sharp(t), 
		\end{equation}
		and
		$\displaystyle\min_{\wt{\boldsymbol{v}}\in E_{{Y}_0}}\normalfont J( \wt{Y},\wt{\boldsymbol{v}})=J(Y^\sharp,\boldsymbol{v}^\sharp)=( \Pc Y_0,Y_0).$
		\item[(c)] The feedback control in \eqref{eqoptcntrl_vel} stabilizes \eqref{eqn: main shiftedVel}. In particular, let us denote the operator $\normalfont\Ac_{\nu,\Pc}:=\Ac_\nu-\Bc\Bc^*\Pc,$ with $\normalfont D(\Ac_{\nu,\Pc})= D(\Ac)$. The semigroup $ \{e^{t\Ac_{\nu,\Pc}}\}_{t\ge 0}$, generated by $(\normalfont\Ac_{\nu,\Pc}, \normalfont D(\Ac_{\nu,\Pc})),$ on $\Hf$ is analytic and exponentially stable, that is, there exist $\gamma>0$ and $M>0$ such that 
		$$\n \|e^{t \Ac_{\nu,\Pc}}\|_{\mathcal{L}(\Hf)}\leq Me^{-\gamma t} \text{ for all }  t>0.$$
	\end{enumerate}
\end{Theorem}

\begin{Remark} \label{rem:F-Inf-Feed-Velo}
	From the above theorem, we define the feedback operator as
	\begin{align} \label{eqfeedbackcntrl}
		K=-\Bc^*\Pc.
	\end{align}
	It encompasses both the finite-dimensional and infinite-dimensional cases, as considered in \cite{WKR}. For our subsequent analysis, the distinction between these two cases is not significant, as the analysis remains valid in both scenarios. Therefore, we simply denote $K$ as the feedback operator without specifying its dimensionality.
\end{Remark}


\section{Full System} \label{sec:FS-NSE}
In this section, we focus on the stabilizability of the non-linear coupled full system by using the feedback gain operator obtained in the case of principal system considered in the previous section. Before proving the stabilizability, let us prove a regularity result required for the corresponding non-homogeneous linear closed loop system:
\begin{equation} \label{eqlinregu}
	\left\{
	\begin{aligned}
		& \widetilde{\zb}_t + \eta A_0\widetilde{\zb} + \kappa A_0\widetilde{\wb}-\nu\widetilde{\zb} = \chi_{\mathcal{O}}K(\widetilde{\zb}, \widetilde{\wb})+ \boldsymbol{f} \text{ in } \Omega\times (0,\infty), \\
		& \widetilde{\wb}_t+\lambda\widetilde{\wb} - \widetilde{\zb} -\nu\widetilde{\wb}= 0 \text{ in } \Omega\times (0,\infty), \\
		& \widetilde{\zb}=0, \quad \& \quad \widetilde{\wb}=0  \text{ on } \Gamma\times (0,\infty), \\
		& \widetilde{\zb}(0)=\zb_0,  \quad \& \quad \widetilde{\wb}(0)=0\text{ in }\Omega.
	\end{aligned}\right.
\end{equation}
where $\wt{\zb}=e^{\nu t}\zb,$ $\wt{\wb}=e^{\nu t}\wb,$ and $K \in \mathcal{L}(\Hf,\Hb)$ is the stabilizing feedback operator defined in \eqref{eqfeedbackcntrl}.

\begin{Theorem}\label{th:est-1}
	Let $\nu\in (0, \nu_0)$ and $K$ be as defined in \eqref{eqfeedbackcntrl}. For any $\zb_0 \in \Vf $ and for any given $\boldsymbol{f}\in L^2(0,\infty;\Hb)$,  the closed loop system \eqref{eqlinregu}
	has a unique solution 
	$(\widetilde{\zb}, \widetilde{\wb})$ such that $\widetilde{\zb}\in L^2(0, \infty; \Hb^2(\Omega))\cap H^1(0, \infty; \Hb)\cap L^\infty(0, \infty;\Vf) $ and $\widetilde{\wb}\in H^1(0, \infty; \Hzz\cap\Hb)$
	satisfying the following estimate
	\begin{align}
		\|\widetilde{\zb}\|_{L^2(0, \infty; \Hb^2(\Omega))}+\|\widetilde{\zb}\|_{H^1(0, \infty; \Hb)}&+\|\widetilde{\zb}\|_{L^\infty(0, \infty;\Vf)}+\|\widetilde{\wb}\|_{H^1(0, \infty; \Hb^2(\Omega))} \\\nonumber &\leq M_1\left(\|\zb_0\|_{\Vf}+\|\boldsymbol{f}\|_{L^2(0, \infty;\Hb)}\right),
	\end{align}
	for some positive constant $M_1$ independent of initial condition $\zb_0$ and forcing term $\fb$.
\end{Theorem}

\begin{proof}
	Note that \eqref{eqlinregu} can be written as 
	\begin{align*}
		\wt{Y}'(t)=(\Ac_\nu + \Bc K)\wt{Y}(t) + \begin{pmatrix} \fb(t) \\ 0 \end{pmatrix} \text{ for all }t>0, \quad \wt{Y}(0)=Y_0. 
	\end{align*}
	Since $\Ac_\nu + \Bc K$ generates an analytic semigroup of negative type on $\Hf$ (see Theorem \ref{th:stb cnt_vel}(c)) and $\begin{pmatrix} \fb(t) \\ 0 \end{pmatrix} \in L^2(0,\infty ; \Hf),$ \cite[Proposition 3.1, Ch. 1, Part II]{BDDM} yields a unique solution $ \wt{Y} \in L^2(0,\infty;\Hf)$ with 
	\begin{align*}
		\|\wt{Y}\|_{L^2(0,\infty; \Hf)} \le C(\|Y_0\|_{\Hf} + \|\fb\|_{L^2(0,\infty; \Hb)}),
	\end{align*}
	for some $C>0$. With $\boldsymbol{p}:=\wt{\zb}+\frac{\kappa}{\eta}\wt{\wb},$ from \eqref{eqlinregu}, we obtain 
	\begin{align}
		& \boldsymbol{p}_t +\eta A \boldsymbol{p} = \Big(\nu+\frac{\kappa}{\eta}\Big)\wt{\zb} + \frac{\kappa}{\eta} (\nu -\lambda)\wt{\wb} +\chi_{\mathcal{O}}K(\wt{\zb}, \wt{\wb}) +\fb \text{ in } \Omega \times (0,\infty),\\
		& \boldsymbol{p}(x,t)=0 \text{ on } \Gamma\times (0,\infty), \\
		& \boldsymbol{p}(x,0)= \zb_0(x) \text{ in }\Omega.
	\end{align}
	Since the right hand size belongs to $L^2(0,\infty; \Lb)$ and $\zb_0 \in \Vf,$ from \cite[Theorem 1.1 and Proposition 1.2, Chapter III]{TemamNSE01}, we find a unique strong solution $\boldsymbol{p} \in L^2(0,\infty; \Hzz \cap \Vf)\cap L^\infty(0,\infty; \Vf) \cap H^1(0,\infty; \Hb)$ such that 
	\begin{align*}
		\|\boldsymbol{p}  \|_{L^2(0,\infty; \Hzz )} & + \|\boldsymbol{p}\|_{L^\infty(0,\infty; \Vf)}+ \|\boldsymbol{p}\|_{H^1(0,\infty; \Hb)} \\
		& \le C\left( \|\zb_0\|_{\Vf}+ \big\|\left(\nu+\frac{\kappa}{\eta}\right)\wt{\zb} + \frac{\kappa}{\eta} (\nu -\lambda)\wt{\wb} +\chi_{\mathcal{O}}K(\wt{\zb}, \wt{\wb}) +\fb \big\|_{L^2(0,\infty; \Hb)}   \right) \\
		& \le  C ( \|\zb_0\|_{\Vf}  + \|\fb\|_{L^2(0,\infty; \Hb)}).
	\end{align*}
	Observe that
	\begin{align*}
		\wt{\wb}_t=-(\lambda + \frac{\kappa}{\eta}-\nu)\wt{\wb} +\boldsymbol{p} \text{ for all }t>0, \text{ and }\wt{\wb}(0)=0,
	\end{align*}
	and recall that $0<\nu<\nu_0=\lambda+\frac{\kappa}{\eta}.$ Therefore, one can write 
	\begin{align*}
		\wt{\wb}(x,t)=\int_0^t e^{-(\lambda + \frac{\kappa}{\eta}-\nu)(t-s)}\boldsymbol{p}(x,s)ds.
	\end{align*}
	Now, estimating this using the regularity of $\boldsymbol{p}$ and the fact that $-(\lambda + \frac{\kappa}{\eta}-\nu)<0,$ we can easily show that $\wt{\wb} \in H^1(0,\infty; \Hzz\cap\Hb)$ and 
	\begin{align*}
		\|\wt{\wb}\|_{H^1(0,\infty;\Hzz)} \le C(\|\zb_0\|_{\Vf} + \|\fb\|_{L^2(0,\infty; \Hb)}),
	\end{align*}
	for some $C>0.$ One can use the regularity of $\wt{\wb}$ and $\boldsymbol{p}$ with the relation $\wt{\zb}=\boldsymbol{p}-\frac{\kappa}{\eta}\wt{\wb}$ to show that $\wt{\zb} $ belongs to the required spaces and satisfies the necessary estimates. We conclude the proof here.
\end{proof}

Now, for any $0< \nu< \lambda<\nu_0,$ 
we are in situation to study the feedback stabilization of the non-linear system. Consider the full shifted system in feedback form as
\begin{equation} \label{eqn:vnon-linear main}
	\left\lbrace
	\begin{aligned}
		& \widetilde{\zb}_t(\cdot, t) +\eta  A_0  \widetilde{\zb}(\cdot, t)+\kappa A_0 \widetilde{\wb}(\cdot,t)-\nu  \widetilde{\zb}(\cdot, t) =\chi_{\mathcal{O}}(\cdot) K( \widetilde{\zb}(\cdot, t), \widetilde{\wb}(\cdot, t)) - e^{-\nu t}B(\widetilde{\zb}, \widetilde{\zb})\\
		& \hspace{3.5cm}  - B(\yb_\infty, \wt{\zb}) - B(\wt{\zb}, \yb_\infty) - \frac{\kappa}{\lambda} e^{-(\lambda-\nu) t} A_0 \yb_\infty \text{ in }\Omega\times(0, \infty),\\
		& \widetilde{\wb}_t(\cdot, t) +\lambda\widetilde{\wb}(\cdot, t)- \widetilde{\zb}(\cdot,t)-\nu \widetilde{\wb}(\cdot,t)=0 \text{ in }\Omega\times(0, \infty),\\
		& \wt{\zb}=0, \quad \& \quad \wt{\wb}=0 \text{ on } \Gamma\times (0,\infty), \\
		&  \widetilde{\zb}(x,0)= \zb_0(x), \; \widetilde{\wb}(0,x)= 0 \quad \text{ in }\Omega
	\end{aligned}\right.
\end{equation}

We prove that the system \eqref{eqn:vnon-linear main} has unique stable solution by using the Banach point fixed theorem. To do that, we first set the space:
\begin{equation}\label{veqspace1}
	\begin{array}{l}
		\mathbb{D}=\Big\lbrace \widetilde{\zb}\in L^2(0,\infty;  \Hb^2(\Omega))\cap C([0,\infty); \Vf)\cap H^1(0,\infty;\Hb), \\[2.mm]
		\mbox{with norm }  
		\|\widetilde{\zb}\|^2_\mathbb{D}= \|\widetilde{\zb}\|^2_{L^2(0,\infty; \Hb^2(\Omega))}+\|\widetilde{\zb}\|^2_{L^\infty(0,\infty; \Vf)}
		+\|\widetilde{\zb}\|^2_{H^1(0,\infty;\Hb)} \Big \rbrace.
	\end{array}
\end{equation}
For any $\rho>0$, we define
\begin{equation}\label{veqspace2}
	\begin{array}{l}
		\mathbb{D}_\rho=\{\widetilde{\zb}\in \mathbb{D}\mid \|\widetilde{\zb}\|_\mathbb{D} \le \rho\}, \, \text{ and }\,
		\mathbb{B}_\rho= \{\widetilde{\wb}\in H^1(0,\infty;  \Hzz\cap\Hb)\mid \|\widetilde{\wb}\|_{H^1(0,\infty; \Hzz)}\le \rho\}.
	\end{array}
\end{equation}
Next, for any $\Psi \in \mathbb{D},$ we define the function
\begin{equation}\label{veqNL1}
	\boldsymbol{f}(\Psi)(x,t)=-e^{-\nu t} B(\Psi,\Psi)  - B(\yb_\infty, \Psi) - B(\Psi, \yb_\infty)- \frac{\kappa}{\lambda} e^{-(\lambda-\nu) t} A_0 \yb_\infty.
\end{equation}

\begin{Lemma} \label{vlem:est}
	Let $\boldsymbol{f}$ be as introduced in \eqref{veqNL1} and $\nu \in (0,\lambda)$ be any number. 
	There exists a positive constant $M_2>0$, such that for all $\Psi, \Psi_1, \Psi_2\in \mathbb{D}$, 
	{\small
		\begin{equation}\label{veqNL2}
			\begin{array}{l}
				(a) \,\|\boldsymbol{f}(\Psi)\|_{L^2(0,\infty; \Hb)}\leq M_2(\|\Psi\|^2_\mathbb{D} +  \|\yb_\infty\|_{\Hzz} +\|\yb_\infty\|_{\Hzz}^2),\\
				(b) \, \|\boldsymbol{f}(\Psi_1)-\boldsymbol{f}(\Psi_2)\|_{L^2(0,\infty; \Hb)}\leq M_2 \Big(\|\Psi_1\|_\mathbb{D}+\|\Psi_2\|_\mathbb{D} + \|\yb_\infty\|_{\Hzz}\Big)\|\Psi_1-\Psi_2\|_\mathbb{D}.
			\end{array}
	\end{equation}}
	Here, the constant $M_2$ depends on $\Omega$  and the Sobolev embedding constant. 
\end{Lemma}

\begin{proof}
	(a) For $\Psi\in \mathbb{D} $ and for all $t>0$, it gives that
	\begin{align}
		\| e^{-\nu t} B(\Psi,\Psi) \|^2_{L^2(0,\infty; \Hb)} & =  \int_0^\infty \|e^{-\nu t} \Pb [(\Psi\cdot\nabla)\Psi]\|_{\Hb}^2 dt  \\
		& \leq  \int_0^\infty e^{-2\nu t}\|\Psi(t)\|^2_{\mathbb{L}^4(\Omega)}\|\nabla\Psi(t)\|^2_{\mathbb{L}^4(\Omega)} dt\nonumber\\ 
		& \leq  C \|\Psi\|_{L^\infty(0,\infty; \Vf)}^2 \|\Psi\|_{L^2(0,\infty; \Hzz)}^2,
	\end{align}
	where we used H\"{o}lder's and  Sobolev's inequalities.  
	Since $\yb_\infty \in \Hzz \cap \Vf,$ we estimate $B(\yb_\infty, \Psi),$  $B(\Psi, \yb_\infty)$ and $ e^{-(\lambda-\nu) t} A_0 \yb_\infty$ as
	{\small
		\begin{align*}
			& \| B(\yb_\infty,\Psi) \|^2_{L^2(0,\infty; \Hb)}  =  \int_0^\infty \| \Pb [(\yb_\infty\cdot\nabla)\Psi]\|_{\Hb}^2 dt\le C \|\yb_\infty\|_{\mathbb{L^\infty}(\Omega)}^2  \|\Psi\|_{L^2(0,\infty; \Vf)}^2,\\
			& \| B(\Psi,\yb_\infty) \|^2_{L^2(0,\infty; \Hb)}  =  \int_0^\infty \| \Pb [(\Psi\cdot\nabla) \yb_\infty]\|_{\Lb}^2 dt\le C\|\yb_\infty\|_{\Hzz}^2 \|\Psi\|_{L^2(0,\infty; \Vf)}^2, \\
			&  \|  e^{-(\lambda-\nu) t} A_0 \yb_\infty\|^2_{L^2(0,\infty; \Hb)}  \le \frac{1}{\lambda-\nu}\|\yb_\infty\|_{\Hzz}^2 \text{ as } 0<\nu<\lambda.
	\end{align*}}
	Now, combing the above estimates, we complete the proof of (a).

	\noindent (b) Addition and subtraction of $e^{-\nu t}B(\Psi_2,\Psi_1)$ followed by a triangle inequality and Sobolev inequality lead to
	\begin{equation}
		\begin{aligned}
			&\|e^{-\nu t}  B(\Psi_1,\Psi_1)  - e^{-\nu t} B(\Psi_2,\Psi_2) \|_{L^2(0,\infty; \Hb)}^2 \\
			&=\|e^{-\nu t}  B(\Psi_1-\Psi_2,\Psi_1)  + e^{-\nu t} B(\Psi_2,\Psi_1-\Psi_2) \|_{L^2(0,\infty; \Hb)}^2 \\
			& \leq  \int_0^\infty \Big(\|\Psi_1(t)-\Psi_2(t)\|^2_{\mathbb{L}^4(\Omega)}\|\nabla\Psi_1(t)\|_{\mathbb{L}^4(\Omega)}^2  +\|\Psi_2(t)\|_{\mathbb{L}^4(\Omega)}^2\|\nabla(\Psi_1-\Psi_2)(t)\|_{\mathbb{L}^4(\Omega)}^2 \Big) dt \nonumber\\
			& \le C \Big( \|\Psi_1-\Psi_2\|_{L^\infty(0,\infty; \Vf)}^2\|\Psi_1\|_{L^2(0,\infty; \Hzz)}^2  + \|\Psi_2\|_{L^\infty(0,\infty; \Vf)}^2\|\Psi_1-\Psi_2\|_{L^2(0,\infty; \Hzz)}^2 \Big).
		\end{aligned}
	\end{equation}
	In a similar way one can have the following estimates:
	\begin{align*}
		\| B(\yb_\infty,\Psi_1) - B(\yb_\infty,\Psi_2) \|^2_{L^2(0,\infty; \Hb)} & =  \int_0^\infty \| \Pb [(\yb_\infty\cdot\nabla)(\Psi_1-\Psi_2)\|_{\Hb}^2 dt \\
		& \le  \|\yb_\infty\|_{\mathbb{L^\infty}(\Omega)}^2 \|\Psi_1-\Psi_2\|_{L^2(0,\infty; \Vf)}^2,  
	\end{align*}
	
	\begin{align*}
		\| B(\Psi_1,\yb_\infty) - B(\Psi_2,\yb_\infty)\|^2_{L^2(0,\infty; \Lb)} 
		\le 	C\|\yb_\infty\|_{\Hzz}^2 \|\Psi_1-\Psi_2\|_{L^2(0,\infty; \Vf)}^2.
	\end{align*}
	A combination of the above estimates with a use of Sobolev embedding leads to the required estimate and the proof is complete. 
\end{proof}

Finally, we state the stabilization result for the non-linear system \eqref{eqn:non-linear intromain}. The closed loop system with the control in feedback form can be written as
\begin{equation} \label{eqn:vnon-linear closed loop main}
	\left\lbrace
	\begin{aligned}
		& \zb_t(\cdot, t) +\eta  A_0  \zb(\cdot, t)+\kappa A_0 \wb(\cdot,t) +B(\zb, \zb) + B(\yb_\infty, \zb) + B(\zb, \yb_\infty) \\
		& \hspace{3cm}- \frac{\kappa}{\lambda} e^{-\lambda t} A_0 \yb_\infty =\chi_{\mathcal{O}}(\cdot) K( \zb(\cdot, t), \wb(\cdot, t))   \text{ in }\Omega\times(0, \infty),\\
		& \wb_t(\cdot, t) +\lambda\wb(\cdot, t)- \zb(\cdot,t)=0 \text{ in }\Omega\times(0, \infty),\\
		& \zb=0, \quad \& \quad \wb=0 \text{ on } \Gamma\times (0,\infty), \\
		&  \zb(x,0)=\zb_0(x), \; \wb(0,x)=0 \quad \text{ in }\Omega,
	\end{aligned}\right.
\end{equation}
where $K(\cdot,\cdot)$ is the linear feedback operator obtained in Remark \ref{rem:F-Inf-Feed-Velo}. 

\begin{Theorem}\label{th:vstab non lin}
	Let $\nu\in(0,\lambda)$ be arbitrary. There exist a continuous linear  operator  
	$K\in \mathcal{L} \left(\Hf,\Hb\right),$ and positive constants 
	$\rho_0$ and $M$ depending on $\eta$, $\lambda$, $\kappa$, such that, for all $0<\rho\le \rho_0$ and for all $\zb_0\in \Vf$  satisfying
	$$\|\zb_0\|_{\Vf}\leq M\rho, \text{ and } \|\yb_\infty\|_{\Hzz} \le M \rho,$$ 
	the non-linear closed loop system \eqref{eqn:vnon-linear closed loop main} admits a unique solution $(\zb,\wb)$ such that \begin{equation}\label{eqn-regu}\zb\in L^2(0, \infty; \Hzz)\cap H^1(0, \infty; \Hb)\cap L^\infty(0, \infty;\Vf),\ \wb \in H^1(0, \infty; \Hzz\cap\Hb),\end{equation}  and
	{\small
		$$ \|e^{\nu\cdot}\zb\|^2_{L^2(0,\infty;\Hb^2(\Omega))}+\|e^{\nu\cdot}\zb\|^2_{L^\infty(0,\infty; \Vf)}+\|e^{\nu\cdot}\zb\|^2_{H^1(0,\infty; \Hb)}+\|e^{\nu\cdot}\wb\|^2_{H^1(0,\infty; \Hb^2(\Omega))}\le 2\rho^2.$$ }
	Moreover, $(\zb, \wb)$ satisfies 
	\begin{align}\label{est_exp} 
		\|(\zb(\cdot,t), \wb(\cdot,t))\|_{\Vf\times \Hb^2(\Omega)} \le C e^{-\nu t} \left( \|\zb_0\|_{\Vf} + \|\yb_\infty\|_{\Hzz} + \|\yb_\infty\|_{\Hzz}^2 \right),
	\end{align}
	for all $t>0$ and for some positive constant $C$ independent of the initial conditions and time $t$. 
\end{Theorem}

\begin{proof}
	Let us first discuss  the well-posedness of the closed loop shifted system \eqref{eqn:vnon-linear main} and then the required result will follow immediately.  To do that, we use the Banach fixed point theorem. For a given $\Psi \in \mathbb{D}_\rho,$ consider the following system: 
	\begin{equation} \label{eqn:vnon-linear main+f}
		\left\lbrace
		\begin{aligned}
			& \widetilde{\zb}_{\Psi_t}(\cdot, t) +\eta  A_0  \widetilde{\zb}_\Psi(\cdot, t)+\kappa A_0 \widetilde{\wb}_\Psi(\cdot,t)-\nu  \widetilde{\zb}_\Psi(\cdot, t) \\
			& \qquad \qquad =\chi_{\mathcal{O}}(\cdot) K( \widetilde{\zb}_\Psi(\cdot, t), \widetilde{\wb}_\Psi(\cdot, t)) + \fb(\Psi)\text{ in }\Omega\times(0, \infty),\\
			& \widetilde{\wb}_{\Psi_t}(\cdot, t) +\lambda\widetilde{\wb}_\Psi(\cdot, t)- \widetilde{\zb}_\Psi(\cdot,t)-\nu \widetilde{\wb}_\Psi(\cdot,t)=0 \text{ in }\Omega\times(0, \infty),\\
			& \wt{\zb}_\Psi=0, \quad \& \quad \wt{\wb}_\Psi=0 \text{ on } \Gamma\times (0,\infty), \\
			&  \widetilde{\zb}_\Psi(x,0)=\zb_0(x), \quad \& \quad \widetilde{\wb}_\Psi(x,0)=0 \, \text{ in }\Omega,
		\end{aligned}\right.
	\end{equation}
	where $\fb$ is as in \eqref{veqNL1}. Since, $\Psi\in \mathbb{D},$ Lemma \ref{vlem:est}(a) implies $\fb(\Psi)\in L^2(0,\infty;\Lb)$ and hence  from Theorem \ref{th:est-1}, we have the existence of a solution $(\wt{\zb}_\Psi, \wt{\wb}_\Psi)$ of \eqref{eqn:vnon-linear main+f} with the estimate
	{\small
		\begin{equation} \label{eq:estttDB}
			\begin{aligned}
				\|\wt{\zb}_\Psi\|_\mathbb{D} + \|\wt{\wb}_\Psi\|_{H^1(0,\infty; \Hzz)} & \le M_1(\|\zb_0\|_{\Vf} + \|\fb(\Psi)\|_{L^2(0,\infty;\Hb)})) \\
				& \le M_1 (\|\zb_0\|_{\Vf} + M_2(\|\Psi\|^2_\mathbb{D} + \|\yb_\infty\|_{\Hzz} +\|\yb_\infty\|_{\Hzz}^2)  ),
			\end{aligned}
	\end{equation} }
	where the last inequality is obtained using Lemma \ref{vlem:est}(a). Now, choosing 
	\begin{align} \label{eq:Vel-rho-cndtn}
		\|\zb_0\|_{\Vf} \le \frac{\rho}{4M_1},\quad \|\yb_\infty\|_{\Hzz} \le \frac{\rho}{4M_1M_2}, \text{ and } 0<\rho \le \min \left\lbrace\frac{1}{4M_1M_2}, 4M_1M_2 \right\rbrace,
	\end{align}
	we obtain 
	\begin{align} \label{eq:Vel-selfmap-cndtn}
		\|\wt{\zb}_\Psi\|_D + \|\wt{\wb}_\Psi\|_{H^1(0,\infty; \Hzz)} & \le \rho. 
	\end{align}
	Let us define the maps $$\mathbb{S}_1:\mathbb{D}_\rho \to \mathbb{D}_\rho \times \mathbb{B}_\rho \ \text{ by }\  \mathbb{S}_1(\Psi)=(\wt{\zb}_\Psi, \wt{\wb}_\Psi)$$  and $$\mathbb{S}_2: \mathbb{D}_\rho \times \mathbb{B}_\rho \to \mathbb{D}_\rho\ \text{ by }\ \mathbb{S}_2(\wt{\zb}_\Psi, \wt{\wb}_\Psi)=\wt{\zb}_\Psi.$$ Then the map $\mathbb{S}:=\mathbb{S}_2 \circ \mathbb{S}_1: \mathbb{D}_\rho \to \mathbb{D}_\rho$ is well defined for all $\rho$ satisfying \eqref{eq:Vel-rho-cndtn}.
	Next, our aim is to show that the map $\mathbb{S}$ is a contraction. For any $\Psi_1, \Psi_2 \in \mathbb{D}_\rho,$ let $(\wt{\zb}_{\Psi_1}, \wt{\wb}_{\Psi_1})$ and $(\wt{\zb}_{\Psi_2}, \wt{\wb}_{\Psi_2})$ be the corresponding solutions pairs. Then $Z=\wt{\zb}_{\Psi_1} - \wt{\zb}_{\Psi_2}$ and $W=\wt{\wb}_{\Psi_1} - \wt{\wb}_{\Psi_2}$ satisfy
	
	\begin{equation} \label{eqn:vnon-linear main+f-diff}
		\left\lbrace
		\begin{aligned}
			& Z_{\Psi_t}(\cdot, t) +\eta  A_0  Z(\cdot, t)+\kappa A_0 W(\cdot,t)-\nu Z(\cdot, t) =\chi_{\mathcal{O}}(\cdot) K( Z(\cdot, t), W(\cdot, t)) \\
			& \hspace{7.3cm} + \fb(\Psi_1)-\fb(\Psi_2)\text{ in }\Omega\times(0, \infty),\\
			& W(\cdot, t) +\lambda W(\cdot, t)- Z(\cdot,t)-\nu W(\cdot,t)=0 \text{ in }\Omega\times(0, \infty),\\
			& Z=0, \quad \& \quad W=0 \text{ on } \Gamma\times (0,\infty), \\
			&  Z(x,0)=0, \quad \& \quad W(0,x)=0 \quad \text{ in }\Omega.
		\end{aligned}\right.
	\end{equation}
	In the above equation, we have used the linearity of $K,$ that is, $$K( \wt{\zb}_{\Psi_1} (\cdot, t), \wt{\wb}_{\Psi_1}(\cdot, t))-K( \wt{\zb}_{\Psi_2}(\cdot, t), \wt{\wb}_{\Psi_2}(\cdot, t))=K( Z(\cdot, t), W(\cdot, t)).$$ Again  for $\rho<\min\left\lbrace\frac{1}{4M_1M_2},4M_1M_2, 2\right\rbrace$, utilization of Theorem \ref{th:est-1} and Lemma \ref{vlem:est} leads to 
	\begin{align*}
		\|Z\|_\mathbb{D} + \|W\|_{H^1(0,\infty; \Hzz)} & \le M_1 \|\fb(\Psi_1)-\fb(\Psi_2)\|_{L^2(0,\infty;\Hb)} \\
		& \le M_1 M_2 \Big(\|\Psi_1\|_\mathbb{D}+\|\Psi_2\|_\mathbb{D} + \|\yb_\infty\|_{\Hzz}\Big)\|\Psi_1-\Psi_2\|_\mathbb{D} \\
		& \le M_1M_2\bigg(\rho + \rho + \frac{\rho}{4M_1M_2}\bigg)\|\Psi_1-\Psi_2\|_\mathbb{D}\\
		& < k \|\Psi_1-\Psi_2\|_\mathbb{D},
	\end{align*}
	for some $0<k<1$. Thus, by choosing any $$0<\rho_0<\min\left\lbrace\frac{1}{4M_1M_2},4M_1M_2, 2\right\rbrace\ \text{ and }\ M=\min\left\lbrace\frac{1}{4M_1}, \frac{1}{4M_1M_2}\right\rbrace,$$ we obtain that $\mathbb{S}$ is a contraction and self map on $\mathbb{D}_\rho$ for all $0<\rho\le \rho_0.$ Hence, by using Banach fixed point theorem, we have a fixed point of $\mathbb{S}$ and hence  a unique solution $(\wt{\zb}, \wt{\wb})$ of \eqref{eqn:vnon-linear main} satisfying 
	\begin{align*}
		\|\wt{\zb}\|_{L^2(0,\infty;\Hzz)}^2+\|\wt{\zb}\|_{L^\infty(0,\infty;\Vf)}^2+\|\wt{\zb}\|_{H^1(0,\infty;\Hb)}^2+\|\wt{\wb}\|_{H^1(0,\infty;\Hzz)}^2 \le 2\rho^2,
	\end{align*}
	for all $0<\rho\le \rho_0,$ $\|\zb_0\|_{\Vf}\le M \rho,$ $\|\fb_\infty\|_{\Hzz}\le M \rho,$ where $M$ and $\rho$ are as mentioned above. On substituting $\zb=e^{-\nu t}\wt{\zb}$ and $\wb=e^{-\nu t}\wt{\wb}$ in \eqref{eqn:vnon-linear main} becomes \eqref{eqn:vnon-linear closed loop main} and therefore \eqref{eqn:vnon-linear closed loop main} admits a unique solution 
	{\small
		$$(\zb, \wb)\in \left\lbrace L^2(0,\infty;\Hzz) \cap L^\infty(0,\infty;\Vf) \cap H^1(0,\infty;\Hb) \right\rbrace \times \left\lbrace H^1(0,\infty;\Hzz\cap\Hb) \right\rbrace$$} satisfying
	{\small
		\begin{align*}
			\|e^{\nu\cdot}\zb\|_{L^2(0,\infty;\Hzz)}^2+\|e^{\nu\cdot}\zb\|_{L^\infty(0,\infty;\Vf)}^2+\|e^{\nu\cdot}\zb\|_{H^1(0,\infty;\Hb)}^2+\|e^{\nu\cdot}\wb\|_{H^1(0,\infty;\Hzz)}^2 \le 2\rho^2.
	\end{align*}}
	To obtain the estimate in \eqref{est_exp}, we take $\Psi=\wt{\zb}_\Psi$ and repeat the estimate in \eqref{eq:estttDB}.
\end{proof}

\begin{Remark} \label{rem:Solv3DNSE2}
	The global solvability (global existence of a strong solution and uniqueness of Leray-Hopf weak solution) of the three-dimensional NSEs remains an open problem, with only partials results are known. In this context, we briefly discuss the local existence and uniqueness result derived from Theorem \ref{th:vstab non lin}. Specifically, by choosing the feedback operator \( K \) as the forcing term and setting \( \yb_\infty = 0 \) in \eqref{eqn:vnon-linear closed loop main}, one can establish a local existence and uniqueness of strong solutions for  NSEs with memory in three dimensions. Furthermore, for the case without memory, setting \( \kappa = 0 \) additionally simplifies the analysis. Hence, it is possible to address the local solvability (for small initial data, the existence and uniqueness of global strong solutions) of the three-dimensional NSEs, both with and without memory.
\end{Remark}

Now, we finally state the result on the main system in the integral form, that is, \eqref{eq:NSE_m_mod}.

\begin{Theorem} \label{th:mainresultStabIntegral}
	Let $\nu \in (0,\lambda).$ Then there exists a linear continuous operator $$\boldsymbol{K}\in \mathcal{L}(L^2(0,\infty;\Hb);\Hb),$$ $\rho_0>0, M>0,$ and $\wt{M}>0$ dependening on $\Omega, \eta, \kappa, \lambda,$ such that for all $0<\rho\le \rho_0,$ $\yb_0 \in \Vf$ and the solution $\yb_\infty\in \Hzz\cap \Vf$ of \eqref{eq:StdNSE} satisfying 
	\begin{align}
		\|\yb_0 -\yb_\infty\|_{\Vf}\le M\rho \text{ and } \|\yb_\infty\|_{\Hzz}\le M\rho,
	\end{align}
	such that a unique solution $\yb$ of 
	\begin{equation} \label{eq:NSE_m_mod-cls-final}
		\left\{
		\begin{aligned}
			& \yb_t -\eta \Delta \yb + (\yb\cdot \nabla) \yb +\nabla p - \kappa \int_0^t e^{-\lambda (t-s)}\Delta \yb(s) ds = \fb_\infty + \chi_{\mathcal{O}}\boldsymbol{K}\yb \text{ in } \Omega\times (0,\infty),\\
			& \nabla\cdot \yb=0 \text{ in } \Omega\times (0,\infty), \\
			& \yb=0  \text{ on } \Gamma\times (0,\infty), \\
			& \yb(0)=\yb_0 \text{ in }\Omega,
		\end{aligned}\right.
	\end{equation}
	satisfies 
	\begin{align}\label{stab y}
		& \|\yb(\cdot,t) -\yb_\infty\|_{\Vf}  + \left\| \int_0^t e^{-\lambda (t-s)} \left(\yb(\cdot,s) ds - \yb_\infty\right)\right\|_{\Hzz}\no \\
		& \qquad \qquad \le \wt{M}  e^{-\nu t}\Big(\|\yb_0-\yb_\infty\|_{\Vf}   + \|\yb_\infty\|_{\Hzz} + \|\yb_\infty\|_{\Hzz}^2\Big), \text{ for all }t>0.
	\end{align}
	Also, the pressure term satisfies
	\begin{align}
		\|p(\cdot,t)-p_\infty(\cdot,t)\|_{L^2(\Omega)} \le \wt{M}e^{-\nu t} \left(\|\yb_0-\yb_\infty\|_{\Vf}  + \|\yb_\infty\|_{\Hzz} + \|\yb_\infty\|_{\Hzz}^2 \right),
	\end{align}
	for all $t>0.$
\end{Theorem}

\begin{proof}
	With the help of $K$ given in Remark \ref{rem:F-Inf-Feed-Velo}, we define $\boldsymbol{K}$ as 
	\begin{align} \label{eq:fdbk-intgrl-NSE}
		\boldsymbol{K}\yb(\cdot, t)= K\left(\yb(\cdot,t), \int_0^t e^{-\lambda (t-s)} \yb(\cdot,s)ds\right) \text{ for all }t>0.
	\end{align}
	Then clearly, $\boldsymbol{K} \in \mathcal{L}(L^2(0,\infty;\Hb);\Hb).$ Then the rest of the proof follows from Theorem \ref{th:vstab non lin} by setting $\wb(\cdot,t)=\int_0^t e^{-\lambda (t-s)}\yb (\cdot,s) ds.$


	\vskip 0.1 cm
	\noindent \textbf{Estimate of the pressure term.} 
	\noindent      Once we have the estimate \eqref{est_exp}, we have from \eqref{eqfeedbackcntrl} and \eqref{eq:fdbk-intgrl-NSE} that 
	\begin{align}\label{exp_u}
		\|\boldsymbol{v}(\cdot,t)\|_{\Hb}\leq M e^{-\nu t} \left( \|\zb_0\|_{\Vf} + \|\yb_\infty\|_{\Hzz} + \|\yb_\infty\|_{\Hzz}^2 \right) \text{ a.e. } t>0.
	\end{align}
	Then following \cite[Thoerem 5.7]{RRS16}, we can estimate pressure $q=p-p_\infty$ in terms of $\zb=\yb-\yb_\infty$ from the eqaution \eqref{eq:NSE_m_mod_z}. Taking the Leray projection $\Pb$ on each term of $\eqref{eq:NSE_m_mod_z}_1$ we get
	\begin{align}\label{pre1}
		\zb_t+A_0\zb-\kappa\int_0^te^{-\lambda(t-s)}A_0\zb(s) ds=\mathcal{G}(\zb,\yb_\infty) \ \text{ in }\ \Hb,
	\end{align}
	where $$\mathcal{G}(\zb,\yb_\infty)=\frac{\kappa}{\lambda}e^{-\lambda t}A_0\yb_\infty+\boldsymbol{v}\chi_{\mathcal{O}}-\Pb[(\zb\cdot\nabla)\zb+(\zb\cdot\nabla)\yb_\infty+(\yb_\infty\cdot\nabla)\zb].$$ Note that the kernel $e^{-\lambda t}$ is positive, that is, $\int_0^T\int_0^te^{-\delta(t-s)}(\zb(s),\zb(t))dsdt\geq 0$ for all $T>0$ and $\zb\in L^2(0,T;\Hb)$. By using this fact and since $\mathcal{G}(\zb,\yb_\infty)\in L^2(0,\infty;\Hb),$ one can obtain the existence of a unique strong solution $\zb\in L^{\infty}(0,\infty;\Vf)\cap L^{2}(0,\infty;\Hzz\cap\Vf)$ to the problem \eqref{pre1}. Since $\boldsymbol{w}=\int_0^te^{-\lambda(t-s)}\zb(s) ds\in L^2(0,\infty;\Hzz\cap\Vf)$ (\eqref{eqn-regu}),  we immediately have $A_0\boldsymbol{w}=\int_0^te^{-\lambda(t-s)}A_0\zb(s) ds\in L^2(0,\infty;\Hb)$. Therefore, rewriting the equation \eqref{pre1}, as $\zb_t+A_0\zb=\kappa A_0\boldsymbol{w}+\mathcal{G}(\zb,\yb_\infty),$ and using the  maximal regularity of the Stokes operator, we find
	\begin{align}\label{eqn-est}
		\|\zb_t\|_{L^2(0,\infty;\Hb)}+\|\zb\|_{L^2(0,\infty;\Hzz)}\leq C\left(\|\boldsymbol{w}\|_{L^2(0,\infty;\Hzz)}+\|\mathcal{G}(\zb,\yb_\infty)\|_{L^2(0,\infty;\Hb)}\right).
	\end{align}
	Therefore using a profound result of de Rham (see \cite[Proposition 1.1 and Remark 1.4]{TemamNSE01},  \cite[Theorem 2.28]{RRS16}), there exist a $q$ such that for a.e. $t\in(0,\infty),$ we have 
	\begin{align*}
		-\nabla q=\zb_t-\eta\Delta\zb-\kappa\Delta\boldsymbol{w}-\mathcal{G}_1(\zb,\yb_\infty),
	\end{align*}
	where 
	$$\mathcal{G}_1(\zb,\yb_\infty)=\frac{\kappa}{\lambda}e^{-\lambda t}\Delta\yb_\infty+\chi_{\mathcal{O}}\boldsymbol{u}-(\zb\cdot\nabla)\zb-(\zb\cdot\nabla)\yb_\infty-(\yb_\infty\cdot\nabla)\zb.$$
	Then using the estimate \eqref{eqn-est}, we further have 
	\begin{align*}
		\|\nabla q\|_{L^2(0,\infty;\Lb)}&\leq C\|\mathcal{G}_1(\zb,\yb_\infty)\|_{L^2(0,\infty;\Lb)}\no\\
		&\leq C\|\zb\|_{L^\infty(0,\infty;\Vf)}\left(\|\zb\|_{L^2(0,\infty;\Hzz)}+\|\yb_\infty\|_{\Hzz}\right)\|\boldsymbol{w}\|_{L^2(0,\infty;\Hzz)}\\&\quad+\|\ub\|_{L^2(0,\infty;\Lb)}.
	\end{align*}
	Then recalling the expression $\zb=\yb-\yb_\infty$ and $q=p-p_\infty$ from the equation \eqref{eq:NSE_m_mod_z} and using \eqref{est_exp}, \eqref{exp_u}, and  \eqref{stab y}, we finally obtain
	$$\|\nabla(p(.,t)-p_\infty)\|_{L^2(\Omega)}\leq \wt{M}  e^{-\nu t}\Big(\|\yb_0-\yb_\infty\|_{\Vf}   + \|\yb_\infty\|_{\Hzz} + \|\yb_\infty\|_{\Hzz}^2\Big),$$
	which completes the proof. 
\end{proof}

\begin{Remark}\label{rem:decayRate-NSE}
	Note that the decay rate obtained in Theorems \ref{th:vstab non lin} - \ref{th:mainresultStabIntegral} is bounded by $\lambda.$ This is due to the estimate in Lemma \ref{vlem:est}(a). However, one can easily observe that the decay rate is better if we study the stabilizability around a zero steady state, in that case the decay rate is $\nu$ for any $\nu \in (0,\nu_0),$ where $\nu_0=\lambda+\frac{\kappa}{\eta},$ as obtained in \eqref{eqlambda0}. In fact, we have the following theorem on stabilizability around zero steady state. 
\end{Remark}

\begin{Theorem} 
	Let $\nu \in (0,\nu_0),$ where $\nu_0$ is as defined in \eqref{eqlambda0}. Then there exists a linear continuous operator $$\boldsymbol{K}\in \mathcal{L}(L^2(0,\infty;\Hb);\Hb),$$ $\rho_0>0, M>0,$ and $\wt{M}>0$ dependening on $\Omega, \eta, \kappa, \lambda,$ such that for all $0<\rho\le \rho_0,$ and $\yb_0 \in \Vf$ satisfying 
	\begin{align}
		\|\yb_0\|_{\Vf}\le M\rho,
	\end{align}
	such that a unique solution $\yb$ of 
	\begin{equation} 
		\left\{
		\begin{aligned}
			& \yb_t -\eta \Delta \yb + (\yb\cdot \nabla) \yb +\nabla p - \kappa \int_0^t e^{-\lambda (t-s)}\Delta \yb(s) ds = \fb_\infty + \chi_{\mathcal{O}}\boldsymbol{K}\yb \text{ in } \Omega\times (0,\infty),\\
			& \nabla\cdot \yb=0 \text{ in } \Omega\times (0,\infty), \\
			& \yb=0  \text{ on } \Gamma\times (0,\infty), \\
			& \yb(0)=\yb_0 \text{ in }\Omega,
		\end{aligned}\right.
	\end{equation}
	satisfies 
	\begin{align}
		& \|\yb(\cdot,t)\|_{\Vf}  + \left\| \int_0^t e^{-\lambda (t-s)} \yb(\cdot,s) ds \right\|_{\Hzz}  \le \wt{M}  e^{-\nu t}\|\yb_0\|_{\Vf}, \text{ for all }t>0.
	\end{align}
	Also, the pressure term satisfies
	\begin{align}
		\|p(\cdot,t)\|_{L^2(\Omega)} \le \wt{M}e^{-\nu t} \|\yb_0\|_{\Vf}  \text{ for all }t>0.
	\end{align}
\end{Theorem}

\appendix


\section{Stabilizability of Vorticity Equation} \label{app:Vort}
We will discuss the stabilizability of vorticity equation in a domain with periodic boundary in two dimensions. In particular, we choose the domain as a two-dimensional torus $\T^2:=[0,2\pi]^2.$
\subsection{The Biot-Savart law} We recall the Biot-Savart law to express the velocity vector field $\yb$ in terms of the vorticity scaler field $w$ (see \cite{mm12}). We have $w=\nabla\times\yb.$ In general (cf. \cite{FZ19}), the Biot-Savart law express the velocity in terms of vorticity as 
\begin{align}\label{VV}
	\yb(x)=(\boldsymbol{k}\ast w)(x)=\int_{\T^2}\boldsymbol{k}(x-y) w(y) \, dy,
\end{align}
where Biot-Savart kernel is given by 
\begin{align}\label{BSK}
	\boldsymbol{k}=\Big(-\frac{\partial G}{\partial x_2}, \frac{\partial G}{\partial x_1}\Big),
\end{align}
and $G$ is the Green function of the Laplacian on the torus with mean zero. We also recall the estimate from \cite[(2.17), (2.18)]{FZ19} as
\begin{align} \label{eq:Linf-LpestT2}
	\|\boldsymbol{k}\ast w\|_{\mathbb{L}^\infty(\T^2)}=\|\yb\|_{\mathbb{L}^\infty(\T^2)}\leq C_p\| w\|_{L^p(\T^2)},
\end{align}
for any $p>2$ and a positive constant $C_p$. Also when $p\geq 1$, $1\leq \alpha<2, \beta\geq 1$ with $\frac{1}{p}+1=\frac{1}{\alpha}+\frac{1}{\beta}$ we infer that 
\begin{align}
	\|\boldsymbol{k}\ast w\|_{\mathbb{L}^p(\T^2)}=\|\yb\|_{\mathbb{L}^p(\T^2)}\leq \|\boldsymbol{k}\|_{\mathbb{L}^\alpha(\T^2)}\|w\|_{L^\beta(\T^2)}.
\end{align}
Reformulating the dynamics \eqref{eq:NSE_m_mod} in terms of vorticity $w=\nabla^\perp\yb = \nabla\times \yb,$ we obtain

\begin{equation} \label{eq:NSEMem_vort}
	\left\{
	\begin{aligned}
		& w_t - \eta \Delta w + (\boldsymbol{k}*w)\cdot \nabla w - \kappa \int_0^t e^{-\lambda(t-s)}\Delta w(s)ds = f_\infty + u\chi_{\mathcal{O}} \text{ in } \mathbb{T}^2\times (0,\infty), \\
		& w(0)=w_0 \text{ in }\mathbb{T}^2,
	\end{aligned}\right.
\end{equation}
where $f_\infty=\nabla \times \fb_\infty$ and $u =\nabla\times \ub.$

We study the feedback stabilization of \eqref{eq:NSEMem_vort} around the steady state obtained by reformulating \eqref{eq:StdNSE} in terms of vorticity $w_\infty=\nabla^\perp \yb_\infty=\nabla \times \yb_\infty$
\begin{equation} \label{eq:stdNSEVort}
	-\eta \Delta w_\infty + (\boldsymbol{k}*w_\infty) \cdot \nabla w_\infty -\frac{\kappa}{\lambda}\Delta w_\infty= f_\infty \text{ in }\T^2.
\end{equation}
Setting $z=w-w_\infty,$ we obtain the system in $z$ as
\begin{equation} \label{eq:NSEMem_vort-cpl}
	\left\{
	\begin{aligned}
		& z_t - \eta \Delta z + (\boldsymbol{k}*z)\cdot \nabla z + (\boldsymbol{k}*z)\cdot \nabla w_\infty + (\boldsymbol{k}*w_\infty)\cdot \nabla z \\
		& \qquad \qquad - \kappa \int_0^t e^{-\lambda(t-s)}\Delta z(s)ds -\frac{\kappa}{\lambda}e^{-\lambda t}\Delta w_\infty=  u\chi_{\mathcal{O}} \text{ in } \mathbb{T}^2\times (0,\infty), \\
		& z(0)=w_0-w_\infty=:z_0 \text{ in }\mathbb{T}^2.
	\end{aligned}\right.
\end{equation}
Here, we follow the technique discussed in \cite{WKR}. Set $v(x,t):=\int_0^t e^{-\lambda(t-s)} z(x,s) ds$, then system \eqref{eq:NSEMem_vort} can equivalently written as  
\begin{equation} \label{eqn:non-linear intromainVort}
	\left\{
	\begin{aligned}
		& z_t - \eta \Delta z +(\boldsymbol{k}*z)\cdot \nabla z + (\boldsymbol{k}*z)\cdot \nabla w_\infty + (\boldsymbol{k}*w_\infty)\cdot \nabla z \\
		& \qquad \qquad \qquad - \kappa \Delta v -\frac{\kappa}{\lambda}e^{-\lambda t}\Delta w_\infty= u\chi_{\mathcal{O}} \text{ in } \mathbb{T}^2\times (0,\infty), \\
		& v_t+\lambda v - z= 0 \text{ in } \mathbb{T}^2\times (0,\infty), \\
		& z(0)=z_0, \quad \& \quad  v(0)=0\text{ in }\mathbb{T}^2.
	\end{aligned}\right.
\end{equation}

\subsection{Principal System}
Before studying the full system, we first focus on the stabilizability of the following system:
\begin{equation} \label{eq:NSEMem_vortLinCpl}
	\left\{
	\begin{aligned}
		& z_t - \nu \Delta z - \kappa \Delta v = u\chi_{\mathcal{O}} \text{ in } \mathbb{T}^2\times (0,\infty), \\
		& v_t+\lambda v - z= 0 \text{ in } \mathbb{T}^2\times (0,\infty), \\
		& z(0)=z_0, \quad  \& \quad v(0)=0\text{ in }\mathbb{T}^2.
	\end{aligned}\right.
\end{equation}
Let us first recall the space $L_{\sharp}^2(\mathbb{T}^2)$ as $L_{\sharp}^2(\mathbb{T}^2)= \left\lbrace z \in L^2(\mathbb{T}^2)\, |\, \int_{\mathbb{T}^2} z dx=0 \right\rbrace$ with the usual $L^2-$ inner product and norm.
Set $\mathcal{H} =L_{\sharp}^2(\mathbb{T}^2)\times L^2_{\sharp}(\mathbb{T}^2)$ equipped with the usual inner product and norm in product space. For $X(t)=\begin{pmatrix} z(\cdot,t) \\ v(\cdot,t) \end{pmatrix},$ the above coupled equation can be re-written as 
\begin{equation} \label{eq:linOp}
	X'(t)= A X(t) + B u(t), \text{ for all }t>0, \quad X(0)= \begin{pmatrix} z_0 \\ 0 \end{pmatrix},
\end{equation}
where $A: D(A) \subset \mathcal{H} \rightarrow \mathcal{H}$ with $D(A)=\{ (z,v) \in \mathcal{H} \, |\, \eta z+\kappa v \in H^2(\mathbb{T}^2)  \}$ is the unbounded operator defined as 
\begin{equation} \label{eqdef:A}
	A \begin{pmatrix} z \\ v \end{pmatrix} = \begin{pmatrix} \Delta (\eta z+\kappa v) \\ -\lambda v+z \end{pmatrix} \text{ for all } \begin{pmatrix} z \\ v \end{pmatrix} \in D(A),
\end{equation}
and the control operator $B : L^2(\mathbb{T}^2) \rightarrow \mathcal{H} $ is defined as 
\begin{equation} \label{eqdef:B-Vort}
	B u= \begin{pmatrix} u \chi_{\mathcal{O}} \\ 0 \end{pmatrix} \text{ for all } u \in L^2(\T^2).
\end{equation}
Next, we calculate the adjoint operator $(A^*, D(A^*))$ (see \cite[Section A.2]{WKRPCE} for a similar calculation) of $(A, D(A))$ as
\begin{align} \label{eqdef:A*}
	A^* \begin{pmatrix} \phi \\ \psi \end{pmatrix} = \begin{pmatrix} \nu\Delta \phi +\psi\\ \kappa \Delta \phi-\lambda\psi \end{pmatrix} \text{ for all } \begin{pmatrix} \phi \\ \psi \end{pmatrix} \in D(A^*):=\left\lbrace \begin{pmatrix} \phi \\ \psi \end{pmatrix} \in \mathcal{H} \, |\, \phi \in H^2(\mathbb{T}^2) \right\rbrace.  
\end{align}
The adjoint operator $B^*:\mathcal{H} \to L^2(\mathbb{T}^2) $ corresponding to $B$ is defined as 
\begin{align} \label{eqdef:adj-B}
	B^* \begin{pmatrix} \phi \\ \psi \end{pmatrix} = \phi\chi_\mathcal{O} \text{ for all } \begin{pmatrix} \phi \\ \psi \end{pmatrix} \in \mathcal{H}.
\end{align}

\subsection{Analytic semigroup and Spectral properties} 
Following \cite{WKR}, below we list down some properties of the operator $A$. 
\begin{Theorem}
	(a) The operator $(A, D(A))$ defined in \eqref{eqdef:A} is densely defined and closed. Furthermore, the half plane $\{\mu \in \mathbb{C} \, | \Re(\mu)\ge 0\} $ is contained in the resolvent set $\rho(A)$ of $A,$  and the following resolvent estimate holds true for some constant $C>0,$ independent of $\mu:$
	\begin{align*}
		\| R(\mu, A)\|_{\mathcal{L}(\mathcal{H})} \le \frac{C}{|\mu|} \text{ for all }\mu (\neq 0) \in \mathbb{C} \text{ with } \Re(\mu) \ge 0.
	\end{align*}
	(b) The operator $(A, D(A))$ generates an analytic semigroup $\{e^{tA}\}_{t\ge 0}$ on $\mathcal{H}.$ \\
	(c) For any $X_0\in \mathcal{H}$ and $ \boldsymbol{f}\in L^2(0,\infty; \mathcal{H}),$ the system 
	\begin{align}
		X'(t)=A X(t) + \boldsymbol{f}(t) \text{ for all }t>0, \quad X(0)=X_0,
	\end{align}
	admits a unique mild solution in $C([0,\infty); \mathcal{H})$ with the representation 
	\begin{align*}
		X(t)=e^{tA}X_0 +\int_0^t e^{(t-s)A} \boldsymbol{f}(s) ds \text{ for all }t>0.
	\end{align*}
\end{Theorem}

To discuss the spectral analysis of the operator $(A, D(A)),$ we first discuss the eigenvalues of the operator $-\Delta$ on the periodic domain $\mathbb{T}^2$. Consider the eigenvalue problem
\begin{equation}
	\begin{aligned}
		-\Delta \phi = \nu_k \phi \text{ in } \mathbb{T}^2.
	\end{aligned}
\end{equation}
Then, there exists (see \cite[Page 166]{RobinIDDS}) an infinite sequence of eigenvalues 
\begin{align} \label{eq:eigvalLapT2}
	\nu_k=|k|^2,
\end{align}
where $k$ is multi-indices in $\mathbb{N}^2.$ The corresponding eigenfunctions are 
\begin{align*}
	\phi_k(x)=\frac{1}{\pi} e^{k\cdot x}.
\end{align*}
Now, on arranging the multi-indices $k$ in such a way that $k_1<k_2$ whenever $|k_1|<|k_2|,$ we obtain 
\begin{align*}
	0<\nu_1\le \nu_2\le \ldots \le \nu_n\le \ldots \rightarrow \infty.
\end{align*}
Using this, we now discuss the spectral properties of the operator $(A, D(A))$ on $\mathcal{H}$ as follows:

\begin{Theorem}[spectral analysis]\label{th:specAna}
	Let $(A, D(A))$ be as defined in \eqref{eqdef:A}. The eigenvalues of $A$ are consists of two sequences $\widetilde{\mu}_k^+$ and $\widetilde{\mu}_k^-,$ where
	\begin{align}
		\widetilde{\mu}_k^\pm = \frac{-(\lambda+\eta \nu_k) \pm \sqrt{(\lambda+\eta \nu_k)^2 - 4\nu_k (\eta \lambda+\kappa)}}{2},
	\end{align}
	where $\nu_k$ are eigenvalues of $-\Delta$ as discussed in \eqref{eq:eigvalLapT2}. Moreover, we have the following properties: 
	\begin{itemize}
		\item[(a)] There are only finitely many complex eigenvalues of $A.$ 
		\item[(b)] The sequence $\widetilde{\mu}_k^+$ converges to $-\nu_0$ as $k\rightarrow \infty$ while the other sequence $\widetilde{\mu}_k^-$ behaves like $-\eta\nu_k$ and goes to $-\infty$ as $k\to \infty,$ where $\nu_0$ is same as defined in \eqref{eqlambda0}.
		\item[(c)] All the eigenvalues have negative real part. In particular, there exists $N_1,N_2\in \mathbb{N}$ such that 
		\begin{align*}
			& \text{ for all }k< N_1, \quad  - \nu_0 < \Re(\widetilde{\mu}_k^+)<0 ,    \quad \text{ for all }k\ge N_1, \Re(\widetilde{\mu}_k^+) \le -\nu_0, \\
			& \text{ for all }k< N_2, \quad  - \nu_0 < \Re(\widetilde{\mu}_k^-)<0 ,    \quad \text{ for all }k\ge N_2, \Re(\widetilde{\mu}_k^-)\le - \nu_0.
		\end{align*}
	\end{itemize}
\end{Theorem}
Note that the set $\{\phi_k\}_{k\in \mathbb{N}}$ is an orthonormal basis of $L^2_\sharp(\mathbb{T}^2).$ With the help of this, one can easily show that the set $\left\lbrace \begin{pmatrix} \phi_k \\0 \end{pmatrix}, \begin{pmatrix} 0 \\ \phi_k \end{pmatrix}, \, k\in \mathbb{N} \right\rbrace $ is an orthonormal basis of $\mathcal{H}= L^2_\sharp(\mathbb{T}^2) \times L^2_\sharp(\mathbb{T}^2).$ Now, we choose eigenfunctions $\xi_k^\pm$ of $A$ corresponding to eigenvalues $\widetilde{\mu}_k^\pm$ as 
\begin{align} \label{eq:eigfun-A}
	\xi_k^\pm = \begin{pmatrix} 1 \\ \frac{1}{\widetilde{\mu}_k^\pm+\lambda} \end{pmatrix}\phi_k \text{ for all }k\in \mathbb{N}.
\end{align}
One can easily see that the set of eigenvalues of $A^*$ as $\left\lbrace \overline{\widetilde{\mu}_k^+}, \overline{\widetilde{\mu}_k^-}, \, |\, k \in \mathbb{N} \right\rbrace$. For all $k\in \mathbb{N},$ we calculate eigenfunctions $\xi_k^{*+}$ and $\xi_k^{*-}$ of $A^*$ corresponding to eigenvalues $\overline{\widetilde{\mu}_k^+}$ and $ \overline{\widetilde{\mu}_k^-},$ respectively as
\begin{align} \label{eq:eigfun-A*}
	\xi_k^{*+} =    \frac{\left(\lambda+ \overline{\widetilde{\mu}_k^+}\right)^2}{\left(\lambda+ \overline{\widetilde{\mu}_k^+}\right)^2-\kappa \nu_k}  \begin{pmatrix} 1 \\ \frac{-\kappa \nu_k}{\lambda+ \overline{\widetilde{\mu}_k^+}} \end{pmatrix} \phi_k \text{ and } \xi_k^{*-} =    \frac{\left(\lambda + \overline{\widetilde{\mu}_k^-}\right)^2}{\left(\lambda+ \overline{\widetilde{\mu}_k^-}\right)^2-\kappa \nu_k}  \begin{pmatrix} 1 \\ \frac{-\kappa \nu_k}{\lambda + \overline{\widetilde{\mu}_k^-}} \end{pmatrix} \phi_k.
\end{align}
This set of families of eigenfunctions of $A$ and $A^*$ satisfies the bi-orthonormality relations.

\begin{Proposition}
	For all $k\in \mathbb{N},$ let $\{ \xi_k^{+} , \xi_k^{-}\} $ and $\{\xi_k^{*+}, \xi_k^{*-}\}$ be sets of eigenfunctions of $A $ and $A^*,$ respectively, as defined in \eqref{eq:eigfun-A} and \eqref{eq:eigfun-A*}, respectively. Then the following bi-orthonormality relations hold:
	\begin{align*}
		& \left( \xi_k^+, \xi_n^{*+}\right)_{L^2_\sharp(\mathbb{T}^2)}  = \delta_k^n, \qquad \left( \xi_k^+, \xi_n^{*-}\right)_{L^2_\sharp(\mathbb{T}^2)} =0 \text{ for all }k, n \in \mathbb{N}, \\
		& \left( \xi_k^-, \xi_n^{*-}\right)_{L^2_\sharp(\mathbb{T}^2)}  = \delta_k^n, \qquad \left( \xi_k^-, \xi_n^{*+}\right)_{L^2_\sharp(\mathbb{T}^2)} =0 \text{ for all }k, n \in \mathbb{N},
	\end{align*}
	where $\delta_k^n$ is as in Proposition \ref{pps:NSEmBiorth}.
\end{Proposition}
Next, following \cite[Section 3.3]{WKR} and proceeding in same line, we conclude the following.
\begin{Theorem}
	The family of eigenfunctions of $A,$ $\{ \xi_k^+, \xi_k^- \, |\, k \in \mathbb{N} \},$ forms a Riesz basis in $\mathcal{H}.$ The same is also true for the family of eigenfunctions of $A^*,$ $\{ \xi_k^{*+}, \xi_k^{*-} \, |\, k \in \mathbb{N} \}.$ Furthermore, the spectrum of $A,$ denoted by $\sigma(A),$ is closure to the set of eigenvalues of $A,$ that is,
	\begin{align}
		\sigma(A):=\text{ closure of } \left\lbrace \widetilde{\mu}_k^+, \widetilde{\mu}_k^- \, |\, k\in \mathbb{N} \right\rbrace \text{ in } \mathbb{C} = \left\lbrace \widetilde{\mu}_k^+, \widetilde{\mu}_k^- \, |\, k\in \mathbb{N} \right\rbrace \cup \{-\nu_0\}.
	\end{align}
\end{Theorem}

\subsection{Stabilizability} 
We first aim to study the open loop stabilizability of \eqref{eq:linOp} with decay $\nu \in (0,\nu_0)$ and hence it is convenient to study the stabilizability of the shifted system:
\begin{equation} \label{eq:siftLinSysOP}
	\wt{X}'(t)=A_\nu \wt{X}(t) + B \wt{u}(t) \quad \text{ for }t>0, \qquad \wt{X}(0)=X_0,
\end{equation}
where $\wt{X}(t)=e^{\nu t}X(t),$ $\wt{u}(t)=e^{\nu t}u(t),$ and
\begin{align} \label{eqdef:A_lambda}
	A_\nu=A+\nu I, \text{ with } D(A_\nu)=D(A).
\end{align}
Note that the spectrum of $A_\nu,$ denoted by $\sigma(A_\nu)$ is 
\begin{align*}
	\sigma(A_\nu) = \{ \widetilde{\mu}_k^++\nu, \widetilde{\mu}_k^-+\nu\, |\, k \in \mathbb{N} \} \cup \{ \nu -\nu_0\}.   
\end{align*}
From Theorem \ref{th:specAna}, it follows that there exists $\tilde{N}_\nu\in \mathbb{N}$ such that
\begin{align} \label{eqref:SpecCond22}
	\Re(\widetilde{\mu}_k^\pm+\nu) <0 \text{ for all }k > \tilde{N}_\nu \text{ and } \Re(\widetilde{\mu}_k^\pm+\nu) \ge 0 \text{ for all }1\le k\le \tilde{N}_\nu .
\end{align}
Consequently, we define
\begin{align} \label{eq:specA_lam+-}
	\sigma_+(A_\nu)=\{\widetilde{\mu}_k^\pm +\nu \, |\, \Re(\widetilde{\mu}_k^\pm +\nu)\ge 0\} \text{ and } \sigma_-(A_\nu)=\{\widetilde{\mu}_k^\pm +\nu \, |\, \Re(\widetilde{\mu}_k^\pm +\nu) < 0\},
\end{align}
and note that $\sigma_+(A_\nu)$ is a finite set.
To obtain the feedback operator, consider the optimal control problem:
\begin{align}\label{eqoptinf_vort}
	\min_{\wt{u}\in E_{X_0}} J( \wt{X},\wt{u}) \text{ subject to \eqref{eq:siftLinSysOP}}, \text{ where } 	J( \wt{X},\wt{u}):=\int_0^\infty\big( \|\wt{X}(t)\|_{\mathcal{H}}^2 + \|\wt{u}(t)\|_{L^2(\T^2)}^2\big) \, dt,
\end{align}
with $E_{X_0}:=\{ \wt{u}\in L^2(0,\infty; L^2(\Omega))\mid \wt{X} \text{ satisfying \eqref{eq:siftLinSysOP} with control }\wt{u} \text{ such that } J( \wt{X},\wt{u}) <\infty\}.$
The next theorem yields the minimizer of \eqref{eqoptinf_vort} as well as the stabilizing control in the feedback form.  

\begin{Theorem}[stabilization for the vorticity case]\label{th:stb cnt}
	Let $\nu \in (0, \nu_0)$ be  any real number. Let $A_\nu$ (resp. $B$) be as defined in \eqref{eqdef:A_lambda} (resp. \eqref{eqdef:adj-B}). Then the following results hold: 
	\begin{enumerate}
		\item[(a)] There exists a unique operator $\normalfont P\in \mathcal{L}(\mathcal{H})$ that satisfies the non-degenerate Riccati equation 
		\begin{equation}\label{eqn:ARE}
			\begin{array}{l}
				A_\nu^*P+P A_\nu-P BB^*P+I=0,\quad P=P^* \geq 0 \text{ on }\mathcal{H}.
			\end{array}
		\end{equation}
		\item[(b)] For any $X_0\in \mathcal{H}$, there exists a unique optimal pair $\normalfont(X^\sharp,u^\sharp)$ for \eqref{eqoptinf_vort}, where for all $t>0$, $X^\sharp(t)$ satisfies the closed loop system 
		\begin{equation}\label{eqcl-loop}
			X{^\sharp}'(t)=(A_\nu-BB^*P)X^\sharp(t),\;\; X^\sharp(0)=X_0,
		\end{equation}
		$u^\sharp(t)$ can be expressed in the feedback form as
		\begin{equation}\label{eqoptcntrl}
			\normalfont u^\sharp(t)= F X^\sharp(t), \text{ with } F=-B^*P,
		\end{equation}
		and
		$\displaystyle\min_{\wt{u}\in E_{{X}_0}}\normalfont J( \wt{X},\wt{u})=J(X^\sharp,u^\sharp)=( P X_0,X_0).$
		\item[(c)] The feedback control in \eqref{eqoptcntrl} stabilizes \eqref{eq:siftLinSysOP}. In particular, let us denote the operator $A_{\nu,P}:=A_\nu-BB^*P,$ with $\normalfont D(A_{\nu,P})= D(A)$. The semigroup $\n \{e^{tA_{\nu,P}}\}_{t\ge 0}$, generated by $(A_{\nu,P}, \normalfont D(A_{\nu,P})),$ on $\n\mathcal{H}$ is analytic and exponentially stable, that is, there exist $\gamma>0$ and $M>0$ such that 
		$$\n \|e^{t A_{\nu,P}}\|_{\mathcal{L}(\mathcal{H})}\leq Me^{-\gamma t} \text{ for all }  t>0.$$
	\end{enumerate}
\end{Theorem}

\begin{Remark}
	As stated in Remark \ref{rem:F-Inf-Feed-Velo}, here also the feedback operator $F$ can be finite- or infinite-dimensional. And, for both the case, rest of our analysis is applicable.
\end{Remark}

\subsection{Stabilizability of Full system}
%

In this section we study the stabilizability of non-linear system \eqref{eqn:non-linear intromainVort}. So we start with a regularity result for linear shifted closed loop system 
\begin{align*}
	\wt{X}'(t)= (A_\nu - B B^*P)\wt{X}(t) +  \begin{pmatrix} f(t) \\ 0 \end{pmatrix}\text{ for all } t>0, \quad \wt{X}(0)=X_0,
\end{align*}
for any forcing term $f\in L^2(0,\infty; L^2(\T^2)).$
Let us set $\widetilde{z}(t)=e^{\nu t}z(t), $ and $ \widetilde{v}(t)=e^{\nu t}v(t)$.

\begin{Theorem} \label{th:regresCls-vort}
	Let $\nu\in (0, \nu_0)$, where $\nu_0$ define in \eqref{eqlambda0}. For any $w_0\in H^1(\mathbb{T}^2)$ and any given $f\in L^2(0, \infty;L^2(\mathbb{T}^2))$ the close loop system 
	\begin{equation} \label{pNSEm}
		\left\{
		\begin{aligned}
			& \widetilde{z}_t - \eta \Delta\widetilde{z} - \kappa \Delta\widetilde{v}-\nu\widetilde{w} = \chi_{\mathcal{O}}F(\widetilde{z}, \widetilde{v})+ f \text{ in } \mathbb{T}^2\times (0,\infty), \\
			& \widetilde{v}_t+\lambda\widetilde{v} - \widetilde{z} -\nu\widetilde{v}= 0 \text{ in } \mathbb{T}^2\times (0,\infty), \\
			& \widetilde{z}(0)=z_0,  \quad \& \quad \widetilde{v}(0)=0\text{ in }\mathbb{T}^2.
		\end{aligned}\right.
	\end{equation}
	has a unique solution $(\widetilde{w}, \widetilde{v})$ such that $\widetilde{w}\in L^2(0, \infty; H^2(\mathbb{T}^2))\cap H^1(0, \infty; L^2_{\sharp}(\mathbb{T}^2))\cap C_b(0, \infty;H^1(\mathbb{T}^2)) $ and $\widetilde{v}\in H^1(0, \infty; H^2(\mathbb{T}^2))$
	satisfying the following estimate
	\begin{align}
		\|\widetilde{w}\|_{L^2(0, \infty; H^2(\mathbb{T}^2))}+\|\widetilde{w}\|_{H^1(0, \infty; L^2_{\sharp}(\mathbb{T}^2))}&+\|\widetilde{w}\|_{L^\infty(0, \infty;H^1(\mathbb{T}^2))}+\|\widetilde{v}\|_{H^1(0, \infty; H^2(\mathbb{T}^2))} \\\nonumber &\leq C(\|w_0\|_{H^1(\mathbb{T}^2)}+\|f\|_{L^2(0, \infty;L^2(\mathbb{T}^2))}),
	\end{align}
	for some positive constant $C$. Here, $F\in\mathcal{L}(\mathcal{H}, U)$ is the the stabilizing feedback operator as obtained in Theorem \ref{th:stb cnt}.
\end{Theorem}
\begin{proof}
	The proof follows exactly the same line as Theorem \ref{th:est-1}.
\end{proof}

For any $0< \nu< \lambda<\nu_0$, we study the stability of the following non-linear system
\begin{equation} \label{eqn:non-linear main}
	\begin{aligned}
		& \widetilde{z}_t(\cdot, t) -\eta \Delta \widetilde{z}(\cdot, t)-\kappa\Delta \widetilde{v}(\cdot,t)-\nu \widetilde{z}(\cdot, t) -\frac{\kappa}{\lambda}e^{-(\lambda-\nu) t}\Delta w_\infty +e^{-\nu t}(\boldsymbol{k}*\wt{z})\cdot \nabla \wt{z}\\
		& \qquad \qquad  + (\boldsymbol{k}*\wt{z})\cdot \nabla w_\infty + (\boldsymbol{k}*w_\infty)\cdot \nabla \wt{z}=\chi_{\mathcal{O}}(\cdot) F(\widetilde{z}(\cdot, t), \widetilde{v}(\cdot, t)) \text{ in }\T^2 \times (0,\infty),\\
		& \widetilde{v}_t(\cdot, t) +\lambda\widetilde{v}(\cdot, t)-\widetilde{z}(\cdot,t)-\nu \widetilde{v}(\cdot,t)=0 \text{ in }Q\\
		& \widetilde{z}(x,0)=z_0(x), \; \widetilde{v}(0,x)=0 \quad \text{ in }\mathbb{T}^2.
	\end{aligned}
\end{equation}
We prove the stabilization result for \eqref{eqn:non-linear main} using the Banach fixed theorem. To do that, we first set the space:
\begin{equation}\label{eqspace1}
	\begin{array}{l}
		D=\Big\lbrace \widetilde{z}\in L^2(0,\infty;  H^2(\mathbb{T}^2))\cap C_b([0,\infty); H^1(\mathbb{T}^2))\cap H^1(0,\infty;L^2_{\sharp}(\Omega)), \\[2.mm]
		\mbox{with norm} \quad 
		\|\widetilde{z}\|^2_D= \|\widetilde{z}\|^2_{L^2(0,\infty; H^2(\mathbb{T}^2))}+\|\widetilde{z}\|^2_{L^\infty(0,\infty; H^1(\mathbb{T}^2))}
		+\|\widetilde{z}\|^2_{H^1(0,\infty;L^2_{\sharp}(\mathbb{T}^2))} \Big \rbrace.
	\end{array}
\end{equation}
For any $\rho>0$, we define
{\small
	\begin{equation}\label{eqspace2}
		\begin{array}{l}
			D_\rho=\{\widetilde{z}\in D\mid \|\widetilde{z}\|_D \le \rho\}, \text{ and }
			\mathcal{B}_\rho= \{\widetilde{v}\in H^1(0,\infty;  H^2(\mathbb{T}^2))\mid \|\widetilde{v}\|_{H^1(0,\infty; H^2(\mathbb{T}^2))}\le \rho\}.
		\end{array}
\end{equation}}
Next, we set 
\begin{equation}\label{eqNL1}
	g(\psi)(x,t)=\frac{\kappa}{\lambda}e^{-(\lambda-\nu) t}\Delta w_\infty - e^{-\nu t}(\boldsymbol{k}*\psi)\cdot \nabla \psi -(\boldsymbol{k}*\psi)\cdot \nabla w_\infty - (\boldsymbol{k}*w_\infty)\cdot \nabla \psi, 
\end{equation} 
for all $(x,t)\in \Omega \times (0,\infty).$ To  apply the fixed point theorem, we need to obtain some suitable estimates of the nonlinear term $g(\psi)$ appearing in \eqref{eqn:non-linear main}. 

\begin{Lemma} \label{lem:est}
	Let $\nu\in (0,\lambda)$ and let $g$ be as introduced in \eqref{eqNL1}. 
	There exists a positive constant $M_3$ depending on $\Omega$, such that for all $\psi, \psi^1, \psi^2\in D$, 
	{\small
		\begin{equation}\label{eqNL2}
			\begin{aligned}
				& (a) \quad \|g(\psi)\|_{L^2(0,\infty; L^2(\mathbb{T}^2))}\leq M_3\left( \|\psi\|^2_D + \|w_\infty\|_{H^1(\T^2)}^2 + \|w_\infty\|_{H^2(\T^2)}  \right),\\
				& (b) \quad \|g(\psi^1)-g(\psi^2)\|_{L^2(0,\infty; L^2_(\mathbb{T}^2))}\leq M_3 \Big(\|\psi^1\|_D+\|\psi^2\|_D + \|w_\infty\|_{H^1(\T^2)}\Big)\|\psi^1-\psi^2\|_D.
			\end{aligned}
	\end{equation}}
\end{Lemma}
\begin{proof}
	(a) Using inequality \eqref{eq:Linf-LpestT2}, we obtain
	\begin{align}
		\|e^{-\nu t}(\boldsymbol{k}*\psi)\cdot \nabla \psi\|^2_{L^2(0,\infty; L^2(\mathbb{T}^2))}\leq & \int_0^\infty e^{-2\nu t}\|\boldsymbol{k}*\psi(t)\|^2_{L^\infty(\Omega)}\|\nabla\psi(t)\|^2_{{L}^2(\mathbb{T}^2)} dt\no\\ \leq & C_p\int_0^\infty e^{-2\nu t}\|\psi\|^2_{L^p(\T^2)} \|\nabla\psi(t)\|^2_{L^2(\mathbb{T}^2)} dt\no\\ \leq & C\|\psi\|^2_{L^\infty(0, \infty, H^1(\mathbb{T}^2))} \|\psi\|^2_{L^2(0, \infty, H^1(\mathbb{T}^2))}.
	\end{align}
	Further, we also have
	\begin{equation}
		\begin{aligned}
			& \|(\boldsymbol{k}*\psi)\cdot \nabla w_\infty\|^2_{L^2(0,\infty; L^2(\T^2))}  \le \int_0^\infty  \|\boldsymbol{k}*\psi(t)\|^2_{L^\infty(\Omega)}\|\nabla w_\infty\|^2_{{L}^2(\mathbb{T}^2)} dt\no\\ 
			& \qquad \quad \qquad \leq C_p\int_0^\infty \|\psi\|^2_{L^p(\T^2)} \|\nabla w_\infty\|^2_{L^2(\mathbb{T}^2)} dt \leq  C\|\nabla w_\infty\|^2_{L^2(\mathbb{T}^2)}\|\psi\|^2_{L^2(0, \infty, H^1(\mathbb{T}^2))},\\
			& \|(\boldsymbol{k}*w_\infty)\cdot \nabla \psi\|^2_{L^2(0,\infty; L^2(\T^2))} 
			\leq  C_p\|\nabla w_\infty\|^2\|\psi\|^2_{L^2(0, \infty, H^1(\mathbb{T}^2))}.
		\end{aligned}
	\end{equation}
	Then, adding the above inequalities the first estimate in Lemma \ref{lem:est} is proved.	
	
	\noindent (b) For the Lipschitz  estimate, note that for all $\psi^1, \psi^2\in D$, 
	\begin{align*}
		&\int_0^\infty e^{-2\nu t}\int_{\mathbb{T}^2}|(k\ast\psi_1)\cdot\nabla\psi_1-(k\ast\psi_2)\cdot\nabla\psi_2|^2\no\\
		&\leq  \int_0^\infty e^{-2\nu t}\int_{\mathbb{T}^2}(|k\ast(\psi_1-\psi_2)|^2|\nabla\psi_1|^2+|k\ast\psi_2|^2|\nabla(\psi_1-\psi_2)|^2)\no\\
		&\leq \int_0^\infty \|k\ast(\psi_1-\psi_2)(t)\|^2_{L^\infty(\mathbb{T}^2)}\|\nabla\psi_1(t)\|^2_{L^2(\mathbb{T}^2)}+\|k\ast\psi_2(t)\|^2_{L^\infty(\mathbb{T}^2)}\|\nabla(\psi_1-\psi_2)(t)\|^2_{L^2(\mathbb{T}^2)}\no\\
		&\leq  C_p\int_0^\infty e^{-2\nu t}\|\psi_1-\psi_2\|^2_{L^p(\mathbb{T}^2)}\|\nabla\psi_1\|^2_{L^2(\mathbb{T}^2)}+\|\psi_2\|^2_{L^p(\mathbb{T}^2)}\|\nabla(\psi_1-\psi_2)\|^2_{L^2(\mathbb{T}^2)}\no\\
		&\leq  C \Big( \|\psi_1\|_{L^\infty(0,\infty; H^1(\T^2))}^2 + \|\psi_2\|_{L^\infty(0,\infty; H^1(\T^2))}^2 \Big) \|\psi_1-\psi_2\|^2_D.
	\end{align*}
	Similarly, we find
	\begin{align*}
		\|(\boldsymbol{k}*\psi_1)\cdot  \nabla w_\infty - (\boldsymbol{k}*\psi_2)\cdot \nabla w_\infty\|^2_{L^2(0,\infty; L^2(\T^2))} 
		& \le \int_0^\infty  \|\boldsymbol{k}*(\psi_1(t)-\psi_2(t))\|^2_{L^\infty(\Omega)}\|\nabla w_\infty\|^2_{{L}^2(\mathbb{T}^2)} dt\no\\ &\leq  C_p\int_0^\infty \|\psi_1-\psi_2\|^2_{L^p(\T^2)} \|\nabla w_\infty\|^2_{L^2(\mathbb{T}^2)} dt \\
		&\leq  C\|\nabla w_\infty\|_{L^2(\T^2)}^2\|\psi_1-\psi_2\|^2_{L^2(0, \infty, H^1(\mathbb{T}^2))},\\
		\|(\boldsymbol{k}*w_\infty)\cdot \nabla \psi_1 - (\boldsymbol{k}*w_\infty)\cdot \nabla \psi_2\|^2_{L^2(0,\infty; L^2(\T^2))} 
		&\leq  C\|\nabla w_\infty\|^2\|\psi_1-\psi_2\|^2_{L^2(0, \infty, H^1(\mathbb{T}^2))}.
	\end{align*}
	Combining all the above estimates, we obtain the second estimate in Lemma \ref{lem:est}.
\end{proof}

Now, we are in a position to get a stabilization result. To obtain that we follow the the technique used in Theorems \ref{th:vstab non lin} - \ref{th:mainresultStabIntegral}. In fact, since we have the regularity result in Theorem \ref{th:regresCls-vort} and Lemma \ref{lem:est}, one can mimic the calculations done in Theorems \ref{th:vstab non lin} - \ref{th:mainresultStabIntegral} and establish the stabilization results on the vorticity equation. Hence, we just state our results without proof.

\begin{Theorem}\label{th:stab non lin}
	Let $\nu$ belong to $(0,\lambda)$. There exist a continuous linear  operator  
	$F\in \mathcal{L} \Big(L^2(\T^2)\times L^2(\T^2) ; L^2(\T^2)\Big),$ and positive constants 
	$\rho_0$ and $M$ depending on $\nu$, $\lambda$, $\kappa$, such that, for all $0<\rho\le \rho_0$ and for all $z_0\in H^1(\mathbb{T}^2)$ and stationary solution $w_\infty \in H^2(\T^2)$ of \eqref{eq:stdNSEVort} satisfying
	$$\|z_0\|_{H^1(\mathbb{T}^2)}\leq M\rho, \text{ and } \, \|w_\infty\|_{H^2(\T^2)}\le M\rho,$$ 
	the non-linear closed loop system 
	\begin{equation} \label{eqn:non-linear mainClsVort}
		\begin{aligned}
			& z_t(\cdot, t) -\eta \Delta z(\cdot, t)-\kappa\Delta v(\cdot,t) -\frac{\kappa}{\lambda}e^{-\lambda t}\Delta w_\infty  +e^{-\nu t}(\boldsymbol{k}*z)\cdot \nabla z + (\boldsymbol{k}*z)\cdot \nabla w_\infty\\
			& \qquad \qquad + (\boldsymbol{k}*w_\infty)\cdot \nabla z=\chi_{\mathcal{O}}(\cdot) F(z(\cdot, t), v(\cdot, t)) \text{ in }\Omega\times (0,\infty),\\
			& v_t(\cdot, t) +\lambda v(\cdot, t)-z(\cdot,t)=0 \text{ in }Q\\
			& z(x,0)=z_0(x), \; v(0,x)=0 \quad \text{ in }\mathbb{T}^2.
		\end{aligned}
	\end{equation}
	admits a unique solution 
	{\small
		$$(w,v)\in \Big(L^2(0, \infty; H^2(\mathbb{T}^2))\cap H^1(0, \infty; L^2_{\sharp}(\mathbb{T}^2))\cap C_b(0, \infty;H^1(\mathbb{T}^2))\Big)\times H^1(0, \infty; H^2(\mathbb{T}^2))$$}  satisfying
	{\small
		$$ \|e^{\nu\cdot}z\|^2_{L^2(0,\infty;H^2(\mathbb{T}^2))}+\|e^{\nu\cdot}z\|^2_{L^\infty(0,\infty; H^1(\mathbb{T}^2))}+\|e^{\nu\cdot}z\|^2_{H^1(0,\infty; L^2(\mathbb{T}^2))}+\|e^{\nu\cdot}v\|^2_{H^1(0,\infty; H^2(\mathbb{T}^2))}\le 2\rho^2.$$ }
	Moreover, $(z,v)$ satisfies 
	$$ \|(z(\cdot,t), v(\cdot,t))\|_{H^1(\mathbb{T}^2)\times H^2(\mathbb{T}^2)} \le C e^{-\nu t} \left( \|z_0\|_{H^1(\mathbb{T}^2)} + \|w_\infty\|_{H^1(\T^2)}^2 + \|w_\infty\|_{H^2(\T^2)}\right),$$
	for all $t>0,$ for some positive constant $C$ independent of initial conditions and $t$. 
\end{Theorem}

In the next theorem, we state the stabilizability result on the vorticity equation \eqref{eq:NSEMem_vort}. 

\begin{Theorem} \label{th:stab non linIntgrl}
	Let $\nu \in (0,\lambda)$ be any given number. Then there exist linear continuous operator $$ F\in \mathcal{L}(L^2(0,\infty;L^2(\T^2));L^2(\T^2)),$$ $\rho_0>0, M>0,$ and $\wt{M}>0$ depending on $\eta, \kappa, \lambda,$ such that for all $0<\rho\le \rho_0,$ $w_0 \in H^1(\T^2)$ and solution $w_\infty\in H^2(\T^2)$ of \eqref{eq:stdNSEVort} satisfying 
	\begin{align}
		\|w_0 -w_\infty\|_{H^1(\T^2)}\le M\rho \text{ and } \|w_\infty\|_{H^2(\T^2)}\le M\rho,
	\end{align}
	there exists a unique solution $w$ of
	{\small 
		\begin{equation*} 
			\left\{
			\begin{aligned}
				& w_t - \eta \Delta w + (\boldsymbol{k}*w)\cdot \nabla w - \kappa \int_0^t e^{-\lambda(t-s)}\Delta w(s)ds = f_\infty + \chi_{\mathcal{O}} \boldsymbol{F}(w) \text{ in } \mathbb{T}^2\times (0,\infty), \\
				& w(0)=w_0 \text{ in }\mathbb{T}^2,
			\end{aligned}\right.
	\end{equation*}}
	such that 
	\begin{align*}
		& \|w(\cdot) -w_\infty\|_{H^1(\T^2)}  + \left\| \int_0^t e^{-\lambda (t-s)} \left( w(s)  - w_\infty \right)ds\right\|_{H^2(\T^2)} \\
		& \qquad \qquad \le \wt{M} e^{-\nu t} \left(\|w_0-w_\infty\|_{H^1(\T^2)} + \|w_\infty\|_{H^1(\T^2)}^2 + \|w_\infty\|_{H^2(\T^2)}\right) \text{ for all }t>0.
	\end{align*}
\end{Theorem}

\begin{Remark}
	As discussed in Remark \ref{rem:decayRate-NSE}, here also we obtain a decay rate bounded by $\lambda.$ However, in the case of stabilizability around a zero steady state, one can get any decay rate $\nu,$ for any $\nu \in (0,\nu_0).$ For more details, we refer to Remark \ref{rem:decayRate-NSE}. 
\end{Remark}

\begin{Remark}
	Here, we briefly discuss the three-dimensional case of vorticity equation which is a vector equation with additional vortex stretching term $-(\boldsymbol{w}\cdot \nabla) \boldsymbol{z}$, where $\boldsymbol{z}$  is the velocity field and $\boldsymbol{w}$ is the vorticity, which is   a vector (see \cite[p. 457]{BarbuVor}) in \eqref{eq:NSEMem_vort} and the stabilization results will be analogous by mimicking the calculations done in the case of (three-dimensional) Navier Stokes equation with memory  in the previous part of this article. For the sake of repetitions, we are not including here.
\end{Remark}

\bibliographystyle{amsplain}
\bibliography{references}

\end{document}